\documentclass[11pt]{article}

\usepackage{amssymb, amsfonts, amsmath, amsthm, graphics}

\newtheorem{thm}{Theorem}[section]
\newtheorem{lemma}[thm]{Lemma}
\newtheorem{prop}[thm]{Proposition}
\newtheorem{cor}[thm]{Corollary}

\theoremstyle{definition}
\newtheorem{defn}[thm]{Definition}
\newtheorem{rem}[thm]{Remark}
\newtheorem{notn}[thm]{Notation}

\newtheorem{exam}[thm]{Example}

\newcommand{\cC}{ {\mathcal C} }

\newcommand{\cD}{ {\mathcal D} }

\newcommand{\Otilda}{ \widetilde{\Omega} }

\newcommand{\cS}{{\cal S}}
\newcommand{\bZ}{ {\mathbb Z} }

\newcommand{\otilda}{ \widetilde{\Omega} }
\newcommand{\ee}{\varepsilon}

\newcommand{\sncb}{ {\cal S}^{(B)}_{nc} }

\newcommand{\nca}{ NC^{(A)} }
\newcommand{\ncb}{ NC^{(B)} }

\newcommand{\ecdef}{ \stackrel{\mathrm{def}}{\Longleftrightarrow} }

\begin{document}

\title{\bf Enumerative properties of 
\boldmath{$\ncb (p,q)$} }
\author{I.P. Goulden
\thanks{Supported by a Discovery Grant from NSERC, Canada.}
\and Alexandru Nica$ \ {  }^{*}$
\and Ion Oancea}

\date{  }

\maketitle

\vspace{-.8cm}

\begin{abstract}
We determine the rank generating function, the zeta polynomial and 
the M\"obius function for 
the poset $NC^{(B)}(p,q)$ of annular non-crossing partitions of 
type $B$, where $p$ and $q$ are two positive integers. 
We give an alternative treatment of some of these results in the case 
$q=1$, for which this poset is a lattice. We also consider the 
general case of multiannular non-crossing
partitions of type $B$, and prove that this reduces to the cases of
non-crossing partitions of type $B$ in the annulus and the disc.
\end{abstract}

\noindent
{\bf AMS classification numbers:} 06A07, 05A15

\vspace{.1in}


\section{Introduction}  \label{section 1}
\setcounter{section}{1}

The enumerative properties of the lattice $NC(n)$ of non-crossing 
partitions of $\{ 1, \ldots , n \}$ have been studied since the 
early 1970's, starting with the paper \cite{K72} of G. Kreweras. 
An important feature of this lattice is its connection to the 
symmetric group $\cS_n$. More precisely, one has a natural poset 
isomorphism
\begin{equation}  \label{eqn:1.1}
NC (n) \simeq [\ee , \alpha_n ] := 
\{ \tau \in \cS_n \mid \ee \leq \tau \leq \alpha_n \} ,
\end{equation}
where ``$\leq$'' is a natural partial order on $\cS_n$, $\ee$ is 
the unit of $\cS_n$, and $\alpha_n$ is the 
long cycle $(1, \ldots , n )$ (see \cite{B97}, \cite{Br01}).

In 1997, V. Reiner \cite{R97} introduced the lattice $\ncb (n)$
of non-crossing partitions of type B. Soon after that (see
\cite{Be03}, \cite{BW02}, \cite{BGN03}) it was noticed that one 
has a poset isomorphism analogous to the one from (\ref{eqn:1.1}):
\begin{equation}  \label{eqn:1.2}
NC^{(B)} (n) \simeq [ \ee , \gamma_n ] 
:= \{ \tau \in B_n \mid \ee \leq \tau \leq \gamma_n \} ,
\end{equation}
where ``$\leq$'' is a natural partial order on the hyperoctahedral
group $B_n$, $\ee$ is the unit of $B_n$, and $\gamma_n$ is the long
cycle $(1, \ldots , n, -1, \ldots , -n)$. (Here $B_n$ is viewed as 
the group of permutations $\tau$ of 
$\{ 1, \ldots , n\} \cup \{ -1, \ldots , -n \}$
that satisfy the condition $\tau (-i) = - \tau (i)$, 
$1 \leq i \leq n$.)

The recent paper \cite{NO07} introduced a family of posets
$\ncb (p,q)$, where $p,q$ are two positive integers. One has a 
poset isomorphism
\begin{equation}  \label{eqn:1.3}
NC^{(B)} (p,q) \simeq [ \ee , \gamma_{p,q} ] \subseteq B_{p+q},
\end{equation}
where the partial order on the hyperoctahedral group $B_{p+q}$ is 
the same as in (\ref{eqn:1.2}), and where 
$\gamma_{p,q}$ is now the permutation with two cycles
\[
\gamma_{p,q} := ( \, 1, \ldots , p, -1, \ldots , -p \, )
( \, p+1, \ldots , p+q, -(p+1), \ldots , -(p+q) \, ) \in B_{p+q}.
\]
The elements of $NC^{(B)} (p,q)$ are certain partitions of the set 
$\{ 1, \ldots , p+q \} \cup \{ -1, \ldots , -(p+q) \}$, and the
partial order considered on $NC^{(B)} (p,q)$ is the one 
given by reverse refinement: $\pi \leq \rho$ if and only 
if every block of $\pi$ is contained in a block of $\rho$.
The distinctive feature of the partitions in $\ncb (p,q)$ is that 
one can draw them as non-crossing diagrams in an {\em annulus} with 
$2p$ points marked on its outside circle and $2q$ points marked 
on its inside circle. (This is unlike the diagrams drawn for 
partitions in $NC^{(B)} (n)$, which are drawn in a {\em disc} with 
$2n$ points marked on its boundary.) 
The poset $\ncb (p,q)$ isn't generally a lattice, but we have a
notable exception occurring in the case when $q=1$. In this case 
the meet operation coincides with the usual ``intersection meet'' 
for partitions -- the blocks of the meet 
$\pi \wedge \rho \in \ncb (p,1)$ are precisely the non-empty 
intersections $A \cap B$ where $A$ is a block of $\pi$ and $B$ is 
a block of $\rho$.

In the present paper we determine the rank generating function, the
zeta polynomial and 
the M\"obius function of the poset $\ncb (p,q)$. Here is how the 
paper is organized. In Section 2 we give a brief review of 
$\ncb (p,q)$ and of a few of its properties that are needed in the 
present paper. Then in Section 3 we discuss 
the special ``lattice'' case $q=1$, when 
the formulas for both the rank generating function and the 
M\"obius function are nicer, and have simpler derivations. It is
amusing to note that $\ncb (n-1, 1)$ has the same rank
generating function as $\ncb (n)$. Nevertheless, one has 
$\ncb (n-1,1) \not\simeq \ncb(n)$ for all $n \geq 3$, as one 
sees by looking at M\"obius functions. 

Section 4 is about the rank generating function of $\ncb (p,q)$
for general $p,q$. We observe that we still have nice formulas 
when we focus on partitions in $\ncb (p,q)$ that have a given 
connectivity (the {\em connectivity} of a partition 
$\pi \in \ncb (p,q)$ is the number of pairs of blocks $A,-A$
of $\pi$ such that $A \neq -A$ and such that $A$ intersects both
sets $\{ \pm 1, \ldots , \pm p \}$ and  
$\{ \pm (p+1), \ldots , \pm (p+q) \}$). But when we just enumerate
the partitions in $\ncb (p,q)$ by their rank we get 1-parameter
sums (which can be summed up to a ``closed form'' when $q=1$, but
not for general $q$). One nice fact that arises in our analysis,
stated as Theorem~\ref{thm:4.4}.3, is that the total
number of partitions in $\ncb (p,q)$ is given by
\begin{equation*}
\left| \ncb (p,q)\right| =\frac{p+q+pq}{p+q}\cdot{2p\choose p}{2q\choose q}.
\end{equation*}

Section 5 is devoted to determining 
the M\"obius function for $\ncb (p,q)$. The method
used here is to count multichains via suitable ``systems 
of parentheses'', on the same lines that were used by Edelman
\cite{E80} to count multichains in $NC(n)$ and then by 
Reiner \cite{R97} to count multichains in $\ncb (n)$. A benefit 
of this approach is that it also yields concrete formulas for the 
zeta polynomial for $\ncb (p,q)$, and for the number of
maximal chains in $\ncb (p,q)$. The formulas obtained are again 
not in closed form, but (again) they can be summed up to closed 
form in the particular case when $q=1$.

In the final Section 6 of the paper, we give a brief description
of the general case of multiannular non-crossing partitions
of type $B$. The main point of the section is to establish that,
due to a topological restriction called the genus inequality,
the general multiannular case reduces in fact to the cases
of non-crossing partitions of type $B$ in a disc or an annulus.

\section{Review of \boldmath{$\ncb (p,q)$} }
\setcounter{section}{2}
\setcounter{equation}{0} 

In this section we review, following \cite{NO07}, a few 
basic facts about the posets $\ncb (p,q)$. 
We will start with a set $\sncb (p,q)$ of ``annular non-crossing 
{\em permutations} of type B'', and we will then define 
$\ncb (p,q)$ in terms of $\sncb (p,q)$.

\begin{defn} \label{def:2.1}
{\em (Partial order on $B_{p+q}$ and the definition of $\sncb (p,q)$.)}

\noindent
We will introduce $\sncb (p,q)$ via a natural partial order on the 
hyperoctahedral group $B_{p+q}$. Let us denote for convenience 
$n:=p+q$. Recall that $B_n$ is the group of permutations $\tau$ of 
$\{ \pm 1, \ldots , \pm n \}$ that satisfy the condition 
$\tau (-i) = - \tau (i)$, $\forall \, 1 \leq i \leq n$. 

$1^o$ We consider the following (non-minimal) set of $n^2$ generators 
of $B_n$:
\[
\{ (i,j)(-i,-j) \mid 1 \leq i,j \leq n, \ i \neq j \}
\cup \{ (i,-j)(-i,j) \mid 1 \leq i,j \leq n, \ i \neq j \}
\]
\begin{equation}  \label{eqn:2.1}
\cup \,  \{ (i, -i) \mid 1 \leq i \leq n \} .
\end{equation}
The generators from (\ref{eqn:2.1}) define a {\em length function} 
$\ell_B$ on $B_n$, as follows: for every $\tau \in B_n$ the length 
$\ell_B ( \tau )$ is the smallest possible 
$k \geq 0$ such that $\tau$ can be factored as a product of $k$ 
generators (with the convention that the product of 0 generators 
is equal to the unit $\ee$ of $B_n$). 

$2^o$ The length function $\ell_B$ satisfies the triangle inequality
\begin{equation}\label{triangineq}
\ \ell_B ( \sigma )  \leq\ell_B ( \tau ) + \ell_B ( \tau^{-1} \sigma ),
\end{equation}
and using the case of equality, we define a
{\em partial order} on $B_n$, where for $\tau , \sigma \in B_n$ 
we put
\begin{equation}  \label{eqn:2.3}
\tau \leq \sigma \ \ecdef
\ \ell_B ( \sigma ) = \ell_B ( \tau ) + \ell_B ( \tau^{-1} \sigma ).
\end{equation}
In other words, the order relation $\tau \leq \sigma$ means that one 
can find minimal factorizations for $\tau$ and for 
$\tau^{-1} \sigma$ into products of generators, such that the
concatenation of these two factorizations gives a minimal 
factorization for $\sigma$.

$3^o$ We define 
\begin{equation}  \label{eqn:2.5}
\sncb (p,q) :=  \{ \tau \in B_n \mid \tau \leq \gamma_{p,q} \} ,
\end{equation} 
where the partial order considered in $B_n$ is the one defined above,
and where $\gamma_{p,q} \in B_n$ is the permutation with two cycles
\begin{equation}  \label{eqn:2.4}
\gamma_{p,q} := ( \, 1, \ldots , p, -1, \ldots , -p \, )
( \, p+1, \ldots , p+q, -(p+1), \ldots , -(p+q) \, ) \in B_n.
\end{equation}
Since the inequality $\ee \leq \tau$ holds for every $\tau \in B_n$,
we thus have that $\sncb (p,q)$ is the  
interval $[ \ee , \gamma_{p,q} ]$ in the group $B_n$.

We mention here that one can give several other equivalent
descriptions for $\sncb (p,q)$. Two such descriptions are 
discussed in \cite{NO07} -- one of them is in terms of a ``genus 
inequality'', and the other is in terms of ``annular 
crossing patterns'' (see Section 2.5 in \cite{NO07}). But these 
alternative descriptions will not be used in the present
paper (although the genus inequality will be used in Section~\ref{sec:6}
for other reasons).
\end{defn}

\begin{defn}    \label{def:2.2}
{\em (Orbit partitions and the definition of $\ncb (p,q)$.)}

\noindent
Let $p,q$ and $n := p+q$ be as above.

$1^o$ For every $\tau \in B_n$ we will use the notation 
$\Omega (\tau)$ for the partition of $\{ \pm 1, \ldots , \pm n \}$ 
into orbits of $\tau$. (Thus two numbers $a,b$ from 
$\{ \pm 1 , \ldots , \pm n \}$ belong to the same block of 
$\Omega ( \tau )$ if and only if there exists $m \in \bZ$ such 
that $\tau^m (a) =b$.) It is obvious that if $A$ is a block of 
$\Omega ( \tau )$ then $-A$ is a block of $\Omega ( \tau )$ as well. 
In the case when $A= -A$ we say that $A$ is {\em inversion-invariant}, 
or that it is a {\em zero-block} of $\Omega ( \tau )$. Clearly, the 
blocks of $\Omega ( \tau )$ that are not inversion-invariant come 
in pairs ($A$ and $-A$, with $A \neq -A$). 

$2^o$ For every $\tau \in B_n$ we will use the 
notation $\otilda ( \tau )$ for the partition of 
$\{ \pm 1, \ldots , \pm n \}$ which is obtained from 
$\Omega ( \tau )$ by grouping together all the inversion-invariant 
blocks of $\Omega ( \tau )$ (if such blocks exist) into one 
block of $\otilda ( \tau )$. That is: if
\[
\Omega ( \tau ) = 
\{ A_1, \ldots , A_k, B_1, -B_1, \ldots , B_l, -B_l \}
\]
with $A_i = - A_i$ for $1 \leq i \leq k$, then
\[
\otilda ( \tau ) = \{ A_1 \cup \cdots \cup A_k, B_1, -B_1, \ldots ,
B_l, -B_l \}.
\]

$3^o$ The set $\ncb (p,q)$ of 
{\em annular non-crossing partitions of type B} is defined as 
\begin{equation}  \label{eqn:2.6}
\ncb (p,q) := \{ \otilda ( \tau ) \mid \tau \in \sncb (p,q) \} .
\end{equation}
\end{defn}

\begin{rem}  \label{rem:2.3}

$1^o$ The set $\ncb (p,q)$ is defined in such a way that the map
$\otilda : \sncb (p,q) \to \ncb (p,q)$ is surjective. It is remarkable 
that this map is in fact a {\em poset isomorphism}, where 
$\sncb (p,q)$ is partially ordered as an interval of $B_{p+q}$ (and 
where $B_{p+q}$ is partially ordered as in Definition 
\ref{def:2.1}.2), while $\ncb (p,q)$ is partially ordered by reverse 
refinement. This is the content of 
Theorem 1.4 in~\cite{NO07}.

$2^o$ From Definition \ref{def:2.2} it is clear that a partition 
$\pi \in \ncb (p,q)$ can never have more than one inversion-invariant
block (if such a block exists, then it is unique).

$3^o$ Let $\widehat{0}$ be the partition of 
$\{ \pm 1, \ldots , \pm (p+q) \}$ into $2(p+q)$ singletons, and 
let $\widehat{1}$ be the partition of 
$\{ \pm 1, \ldots , \pm (p+q) \}$ that has only one block. 
Then $\widehat{0}, \widehat{1} \in \ncb (p,q)$, as it is
clear that $\widehat{0} = \Otilda ( \ee )$ and 
$\widehat{1} = \Otilda ( \gamma_{p,q} )$ (where $\ee$ is the unit
of $B_{p+q}$, while $\gamma_{p,q}$ is as in Equation 
(\ref{eqn:2.4})). The partitions $\widehat{0}$ and 
$\widehat{1}$ are the minimal and maximal elements, respectively, 
of the poset $\ncb (p,q)$.
\end{rem}
 
\begin{rem}    \label{rem:2.4}
{\em (Rank and connectivity for a partition in $\ncb (p,q)$.)}

$1^o$ It is immediate that $\sncb (p,q)$ is a ranked poset, where 
the rank of a permutation $\tau \in \sncb (p,q)$ is given by
the length $\ell_B ( \tau )$ from Definition \ref{def:2.1}.1.
It is moreover not hard to see that 
$\ell_B ( \tau )$ can be alternatively described in terms of the 
cycle structure of $\tau$, by the formula
\begin{equation}  \label{eqn:2.2}
\ell_B ( \tau ) = (p+q) - \frac{1}{2} \cdot \Bigl( \# \mbox{ of 
orbits $A$ of $\tau $ such that $A \neq -A$} \Bigr) .
\end{equation} 
As a consequence, we see that $\ncb (p,q)$ is a ranked poset 
as well, where the rank of a partition $\pi \in \ncb (p,q)$ is 
given by the formula
\begin{equation}  \label{eqn:2.7}
\mbox{rank} ( \pi ) = (p+q) - \frac{1}{2} \cdot \Bigl( \# \mbox{ of 
blocks of $\pi$ that are not inversion-invariant} \Bigr) .
\end{equation}

$2^o$ Another important statistic for partitions in $\ncb (p,q)$ 
is connectivity. For $\pi \in \ncb (p,q)$ we will call 
{\em connectivity} of $\pi$ the number
\begin{equation}  \label{eqn:2.8}
c := \frac{1}{2} \left( \begin{array}{c}
\mbox{\# of blocks $A$ of $\pi$ such that $A \neq -A$}  \\
\mbox{and such that $A$ intersects both sets}           \\
\mbox{$\{ \pm 1, \ldots , \pm p \}$ and 
               $\{ \pm (p+1), \ldots , \pm (p+q) \}$ }
\end{array}  \right) .
\end{equation}
An important fact concerning the concept of connectivity is 
the following: 
\begin{equation}  \label{eqn:2.9}
\left\{  \begin{array}{l}
\mbox{if $\pi \in \ncb (p,q)$ has connectivity $c > 0$,}  \\
                                                             \\
\mbox{then $\pi$ has no inversion-invariant blocks}
\end{array} \right.
\end{equation}
(see \cite{NO07}, Proposition 3.4). 
Thus the blocks of a partition $\pi$ with connectivity $c > 0$
all come in pairs $A, -A$ with $A \neq -A$; there are $c$ pairs of 
blocks as in (\ref{eqn:2.8}), while each of the 
remaining pairs is either ``exterior'' 
($A, -A \subseteq \{ \pm 1, \ldots , \pm p \}$)
or ``interior'' 
($A, -A \subseteq \{ \pm (p+1), \ldots , \pm (p+q) \}$). Note 
moreover that if $c>0$ and if $e$ and $i$ denote, respectively,
the number of exterior and of interior pairs of blocks of $\pi$, 
then one has the inequalities:
\begin{equation}  \label{eqn:2.10}
\left\{  \begin{array}{l}
1 \leq c \leq \min \{ p,q \} , \mbox{ and }   \\
                                              \\
0 \leq e \leq p-c, \ \ 0 \leq i \leq q-c.
\end{array}  \right.
\end{equation}
\end{rem}

\begin{rem}    \label{rem:2.5}
It is instructive at this point to give a brief discussion, 
based on connectivity,
about how the adjusted orbit map $\otilda : \sncb (p,q) \to \ncb (p,q)$ 
works. Let $\pi$ be a partition in $\ncb (p,q)$, and let $c$ be the 
connectivity of $\pi$. There are two possible cases. 

\noindent 
(a) $c > 0$. Then by the fact stated in (\ref{eqn:2.9}) above we have 
$\pi = \Omega ( \tau ) = \otilda ( \tau )$,
where $\tau$ is a (uniquely determined) permutation in $\sncb (p,q)$,
and $\tau$ has no inversion-invariant orbits. 

\noindent
(b) $c=0$. Let $\tau$ denote the unique 
permutation in $\sncb (p,q)$ such that $\otilda ( \tau ) = \pi$. 
Then every orbit of $\tau$ is contained 
either in $\{ \pm 1, \ldots , \pm p \}$ or in 
$\{ \pm (p+1), \ldots , \pm (p+q) \}$ (see Lemma 3.3 of \cite{NO07}).
Moreover, $\tau$ can have at most one inversion-invariant orbit 
contained in $\{ \pm 1, \ldots , \pm p \}$, and at most 
one inversion-invariant 
orbit contained in $\{ \pm (p+1), \ldots , \pm (p+q) \}$ (this is 
due to the fundamental fact from \cite{R97} that partitions in
$\ncb (p)$ or $\ncb (q)$ can have at most one zero-block). If $\tau$
has two inversion-invariant orbits, then $\pi$ is obtained from the 
orbit partition $\Omega ( \tau )$ by joining together these two 
orbits; otherwise (if $\tau$ has at most one inversion-invariant orbit) 
we just have $\pi = \Omega ( \tau )$. 

The case (b) of the discussion was the more complicated one to 
describe, but one should keep in mind that typically this is the 
simpler case to handle. Indeed, the case (b) can be summarized as 
follows: if $\pi \in \ncb (p,q)$ has connectivity equal to 0, then 
$\pi$ is obtained by ``putting together'' a partition 
$\pi_{ext} \in \ncb (p)$ and a partition $\pi_{int} \in \ncb (q)$, 
with a special rule for 
what to do when both $\pi_{ext}$ and $\pi_{int}$ have zero-blocks.
\end{rem}

\begin{rem}  \label{rem:2.6}
We conclude this section with a comment on ``how to draw pictures'' 
of partitions in $\ncb (p,q)$. In fact, what one does 
is to draw (equivalently) pictures of permutations in $\sncb (p,q)$.
In order to do this, one starts by representing the elements of
$\{ \pm 1, \ldots , \pm (p+q) \}$ as points on the boundary of an
annulus: on the outside circle of the annulus we mark $2p$ points 
which we label clockwise as $1, \ldots, p, -1, \ldots , -p$ (in 
this order), and on the inside circle of the annulus we mark $2q$
points which we label counterclockwise as 
$p+1, \ldots, p+q, -(p+1), \ldots , -(p+q)$ (in this order).
In terms of pictures drawn in this annulus, the fact that a
permutation $\tau \in B_{p+q}$ belongs to $\sncb (p,q)$
corresponds then to the following prescription: one can draw a 
closed contour for each of the cycles of $\tau$, such that

\noindent
(i) each of the contours does not self-intersect, and goes clockwise
around the region it encloses;

\noindent
(ii) the region enclosed by each of the contours is contained in the
annulus;

\noindent
(iii) regions enclosed by different contours are mutually disjoint.

Two concrete examples of such drawings are given 
in Figure 1 below, in the particular case when $p=4$ and $q=2$.
On the left we have the drawing of the permutation 
\[
\tau_1 = (1,2,5)(-1,-2,-5)(3,-6)(-3,6)(4)(-4) \in \sncb (4,2);
\]
the partition corresponding to it is 
\[
\pi_1 = \Omega ( \tau_1 ) = \Otilda ( \tau_1 ) 
= \Bigl\{ \, \{ 1,2,5 \}, \{ -1,-2,-5 \}, \{ 3,-6 \},
\{ -3, 6 \}, \{ 4 \}, \{ -4 \} \, \Bigr\},
\]

\begin{center}
\scalebox{.72}{\includegraphics{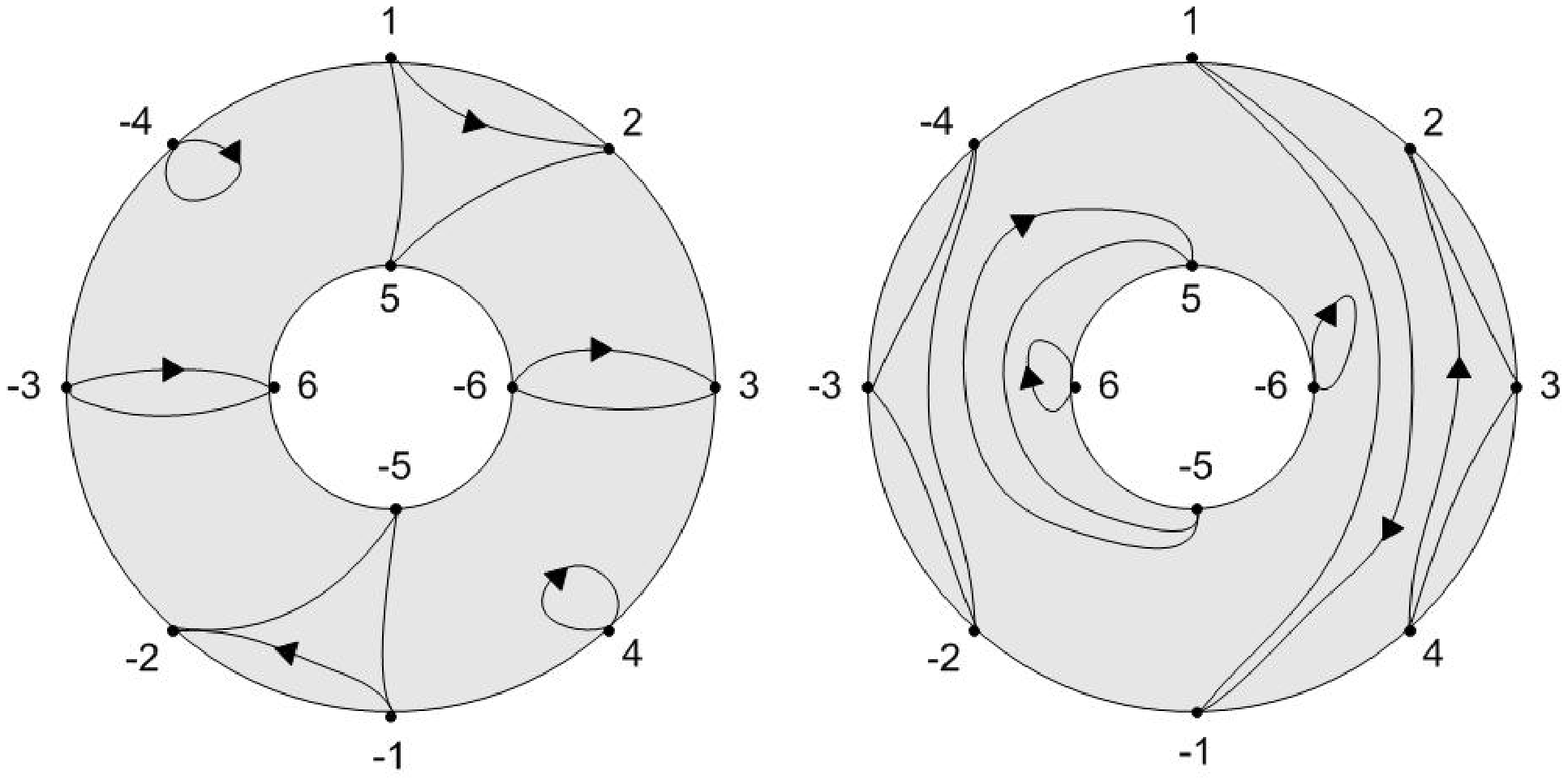}}

\vspace{14pt}

{\bf Figure 1.} Examples of pictures of
permutations in $\sncb (4,2)$.
\end{center}

\noindent
which has connectivity $c=2$. On the right of Figure 1 we have the 
drawing of the permutation 
\[
\tau_2 = (1,-1)(2,3,4)(-2,-3,-4)(5,-5)(6)(-6) \in \sncb (4,2);
\]
the partition corresponding to it is 
\[
\pi_2 = \Otilda ( \tau_2 ) 
= \Bigl\{ \, \{ 1,-1,5,-5 \}, \{ 2,3,4 \}, \{ -2,-3,-4 \}, 
\{ 6 \}, \{ -6 \} \, \Bigr\},
\]
which has connectivity $c=0$. Note that in the latter example we 
have $\Otilda ( \tau_2 ) \neq \Omega ( \tau_2 )$, since the 
inversion-invariant block $\{ 1,-1,5,-5 \}$ of 
$\Otilda ( \tau_2 )$ is obtained by joining together the two 
inversion-invariant orbits of $\tau_2$.
\end{rem}

\section{Rank cardinalities and M\"obius function for 
\boldmath{$\ncb (n-1,1)$} }
\setcounter{section}{3}

Whereas the poset $\ncb (p,q)$ isn't a lattice in general, it is 
\setcounter{equation}{0} 
nevertheless true that $\ncb (n-1,1)$ is a lattice for every 
$n \geq 2$; and moreover, 
the meet operation on $\ncb (n-1,1)$ coincides with the usual 
``intersection meet'' for partitions -- the blocks of the meet 
$\pi \wedge \rho \in \ncb (n-1,1)$ are precisely the non-empty 
intersections $A \cap B$ where $A$ is a block of $\pi$ and $B$ is 
a block of $\rho$. For a proof of these facts, see
Theorem 1.5 of \cite{NO07}.
The present section is devoted to this special ``lattice'' case,
when the formulas for both the rank generating function and the 
M\"obius function are nicer, and can be easily derived from known
facts about $NC(n)$ and $\ncb (n)$.

The rank cardinalities for $\ncb (n-1,1)$ will be presented in
Theorem \ref{thm:3.2}. We first record a few known facts 
that will be used in the proof of this theorem.

\begin{rem}  \label{rem:3.1}

$1^o$ We will use the well-known binomial identity
\begin{equation} \label{vanderm}
\sum_{k=0}^{n-r} {n \choose k}{n \choose k+r} ={2n \choose n-r}
\end{equation}
for any integers $0 \leq r \leq n$. This is a special case of
the Chu-Vandermonde identity (see
for instance Corollary 2.2.3 on page 67 of \cite{AAR99}).

$2^o$ We will use the rank generating functions for the posets 
$NC(n) \Bigl( \, = \nca (n) \, \Bigr)$ and $\ncb (n)$. 

(A) The rank of a partition $\pi \in \nca (n)$ is given by the
formula 
\[
\mbox{rank}( \pi ) = n - \left( \# \mbox{ of blocks of } \pi
\right) .
\]
For every $0 \leq k \leq n-1$, we have (see Corollary~4.1 
of~\cite{K72}) that 
\begin{equation}  \label{eqn:3.01}
\vline \ \Bigl\{ \pi \in \nca (n) \mid \mbox{rank} ( \pi ) = k 
         \Bigr\} \ \vline \ = \ \frac{1}{n} 
\left( \begin{array}{c} n \\ k \end{array} \right)
\left( \begin{array}{c} n \\ k+1 \end{array} \right) .
\end{equation}
The numbers appearing on the right-hand side of 
(\ref{eqn:3.01}) are called {\em Narayana numbers}. The total
number of partitions in $\nca (n)$ is the Catalan number
\begin{equation}  
\vline \,  \nca (n) \, \vline \ = \ \frac{1}{n+1} {2n \choose n}.
\end{equation}

(B) The rank of a partition $\pi \in \ncb (n)$ is given by the
formula 
\[
\mbox{rank}( \pi ) = n - \frac{1}{2} \left( \begin{array}{l}
\# \mbox{ of blocks $A$ of } \pi    \\
\mbox{such that $A \neq -A$}
\end{array}  \right) .
\]
For every $0 \leq k \leq n$, we have (see Proposition 6 of 
\cite {R97}) that 
\begin{equation}  \label{eqn:3.02}
\vline \ \Bigl\{ \pi \in \ncb (n) \mid \mbox{rank} ( \pi ) = k 
         \Bigr\} \ \vline \ = \ 
{n \choose k}^2.
\end{equation}
The total number of partitions in $\ncb (n)$ is
\begin{equation}  \label{eqn:3.021}
\vline \, \ncb (n) \, \vline \ = \ {2n \choose n}.
\end{equation}

$3^o$ We will use a natural ``absolute value map'' that sends
$\ncb (n)$ to $\nca (n)$. We start with the map
$\mbox{Abs} : \{ \pm 1, \ldots , \pm n \} \to \{ 1, \ldots , n \}$
that sends $\pm i$ to $i$, for every $1 \leq i \leq n$. Note that
for every $\pi \in \ncb (n)$ it makes sense to consider the partition 
of $\{ 1, \ldots , n \}$ into blocks of the form Abs$(B)$, with $B$ 
a block of $\pi$; this partition of $\{ 1, \ldots , n \}$ will be 
denoted by ``Abs$( \pi )$''. It turns out that 
$\mbox{Abs}( \pi ) \in \nca (n)$ for every $\pi \in \ncb (n)$, and 
moreover, that the map 
\begin{equation}  \label{eqn:3.25}
\ncb (n) \ni \pi \mapsto \mbox{Abs} ( \pi ) \in \nca (n)
\end{equation}
defined in this way is an $(n+1)$-to-1 map (see Section 1.3
of \cite{BGN03}). In the proof of the next theorem we will use the 
following property (also observed in Section 1.3 of \cite{BGN03})
of the map Abs from (\ref{eqn:3.25}):
\begin{equation}  \label{eqn:3.26}
\left\{  \begin{array}{l} 
\mbox{Given a partition $\pi_o \in \nca (n)$ and a block 
                                         $A$ of $\pi_o$}   \\
\mbox{there exists a unique $\pi \in \ncb (n)$ with a 
                                           zero-block $Z$} \\
\mbox{such that Abs$( \pi ) = \pi_o$ and Abs$(Z) = A$.}
\end{array}  \right.
\end{equation}
\end{rem}

\begin{thm}  \label{thm:3.2}
Let $n \geq 2$ be an integer. Then 
\begin{equation}  \label{eqn:3.1}
\vline \ \ncb (n-1,1) 
\ \vline \ = \ 
\left( \begin{array}{c} 2n \\ n \end{array} \right),
\end{equation}
and for every $0 \leq k \leq n$ we have that
\begin{equation}  \label{eqn:3.2}
\vline \ \Bigl\{ \pi \in \ncb (n-1,1) \mid \mbox{rank} ( \pi ) 
= k \Bigr\} \ \vline \ = \ 
\left( \begin{array}{c} n \\ k \end{array} \right)^2.
\end{equation}
\end{thm}

\begin{proof} Equation (\ref{eqn:3.1}) follows from (\ref{eqn:3.2})
and (\ref{vanderm}), hence it will suffice to verify~(\ref{eqn:3.2}).
We fix a $k$, for which we will prove (\ref{eqn:3.2}). We 
will assume $k \neq 0$ (the case $k=0$ is obvious).

{}From the first inequality (\ref{eqn:2.10}) in Remark \ref{rem:2.4}
it is clear that every partition in $\ncb (n-1,1)$
has connectivity equal to 0 or 1. Let us denote
\begin{equation}  \label{eqn:3.3}
\left\{  \begin{array}{ccl}
\cC & := & \{ \pi \in \ncb (n-1,1) \mid 
             \pi \mbox{ has rank $k$ and connectivity } 1 \} ,   \\
    &                                                            \\
\cD & := & \{ \pi \in \ncb (n-1,1) \mid 
             \pi \mbox{ has rank $k$ and connectivity } 0 \} . 
\end{array}   \right.
\end{equation}
We note that every partition $\pi \in \cD$ must be of the form 
$\pi = \otilda ( \tau )$, where $\tau$ is a permutation in 
$\sncb (n-1,1)$ that leaves invariant the set $\{ n, -n \}$. Clearly, 
there are only two possibilities 
for how $\tau$ can act on $\{ n, -n \}$: either $\tau (n) = n$
and $\tau (-n) = -n$, or $\tau (n) = -n$ and $\tau (-n) = n$. We will 
denote by $\cD_{+}$ and $\cD_{-}$, respectively, the set of partitions 
$\pi \in \cD$ for which the first (respectively the second) of these 
possibilities occurs. We thus have $\cD = \cD_{+} \cup \cD_{-}$, 
disjoint, and it is clear that
\begin{equation}  \label{eqn:3.4}
\vline \ \{ \pi \in \ncb (n-1,1) \mid \mbox{rank}(\pi) = k \} 
\ \vline \ = \ \vline \, \cC \ \vline \ + \ 
\vline \, \cD_{+} \, \vline \ + \ 
\vline \, \cD_{-} \, \vline \ .
\end{equation}

It is immediate to see that $\cD_{+}$ and in $\cD_{-}$ are in 
bijection with the sets of partitions in $\ncb (n-1)$ that have
rank equal to $k$, and respectively $k-1$. (For instance for 
$\cD_{-}$ we observe that every $\pi \in \cD_{-}$ is canonically
obtained from a partition $\pi_o$ of rank $k-1$ in $\ncb (n-1)$,
as follows:
if $\pi_o$ has no zero-block then we add to it a 2-element block 
$\{ n, -n \}$, while if $\pi_o$ has a zero-block $Z$ then we 
replace $Z$ by $Z \cup \{ n , -n \}$.) By taking (\ref{eqn:3.02})
into account, we thus find that
\[
\mid  \cD_{+}  \mid \ = {n-1\choose k}^2 \ \ \mbox{ and }
\mid \cD_{-} \mid \ = { n-1 \choose k-1 }^2.
\]

Let us now count the partitions in the set $\cC$ from
(\ref{eqn:3.3}). Let $\pi$ be in $\cC$, and let us denote by $A$ the 
block of $\pi$ that contains $n$. We know that $A \neq -A$, 
and that $A \cap \{ \pm 1, \ldots , \pm (n-1) \} \neq \emptyset$.
Let $\pi_o$ be the partition of $\{ \pm 1 , \ldots , \pm (n-1) \}$ 
that is obtained from $\pi$ by taking its blocks $A$ and $-A$ and 
replacing them with just one block, 
$Z := \bigl( A \cup (-A) \bigr) \setminus \{ n , -n \}$.
It is immediately seen that $\pi_o \in \ncb (n-1)$, and that the 
rank of $\pi_o$ in $\ncb (n-1)$ is equal to $k$. The partition 
$\pi$ we started with cannot be uniquely retrieved from $\pi_o$,
but a moment's thought shows that $\pi$ {\em can} be uniquely 
retrieved from the pair $( \pi_o , \tau (n) )$, where 
$\tau \in \sncb (n-1,1)$ is the permutation that corresponds to 
$\pi$. (The number $m= \tau (n) \in \{ \pm 1, \ldots , \pm (n-1) \}$ 
could simply be described as ``the point of $A$ that follows $n$'', 
when we move around $A$ in clockwise order.) 

The observations made in the preceding paragraph give us
a one-to-one map
\begin{equation}  \label{eqn:3.5}
\cC \ni \pi \mapsto \Bigl( \pi_o , \tau (n) \Bigr) \in
\Bigl\{ \, ( \pi_o , m ) \begin{array}{ll}
\vline & \pi_o \in \ncb (n-1) \mbox{ of rank $k$ and}  \\
\vline & \mbox{ with zero-block $Z$, and $m \in Z$}
\end{array} \Bigr\} .
\end{equation}
It is quite easy to see that the map in (\ref{eqn:3.5}) is surjective 
as well. In pictorial terms: given $\pi_o \in \ncb (n-1)$ with 
zero-block $Z$, and given an element $m \in Z$, we always know how to 
deform the convex polygon enclosed by $Z$ so that it becomes a union 
of three regions -- a small disc (which is part of a newly created 
annulus), and two regions enclosed by blocks $A, -A$ of a partition
$\pi \in \ncb (n-1,1)$. The role of the element $m \in Z$ in this 
geometric construction is to determine what side of the convex polygon
enclosed by $Z$ has to be deformed, and to indicate where on the 
emerging small disc we should place the labels $n$ and $-n$.

Let us next observe that by using the ``Abs'' map and its property 
reviewed in (\ref{eqn:3.26}) of Remark \ref{rem:3.1}.2, we get another 
bijection
\begin{equation}  \label{eqn:3.6}
( \pi_o , m ) \mapsto \Bigl( \, \mbox{Abs}( \pi_o ), m \, \Bigr) ,
\end{equation}
that sends the set
$\Bigl\{ \, ( \pi_o , m ) \begin{array}{ll}
\vline & \pi_o \in \ncb (n-1) \mbox{ of rank $k$ and}  \\
\vline & \mbox{ with zero-block $Z$, and $m \in Z$}
\end{array} \Bigr\}$
onto the Cartesian product
$\bigl\{ \rho \in \nca (n-1) \mid \mbox{rank} ( \rho ) = k-1 \bigr\} 
\times \{ \pm 1, \ldots , \pm (n-1) \}$.
By using the bijections (\ref{eqn:3.5}) and (\ref{eqn:3.6}) we thus 
find that 
\begin{align*}
\vline \, \cC \, \vline \
& = \ \vline \ \Bigl\{ \rho \in \nca (n-1) \mid 
     \mbox{rank} ( \rho ) = k-1 \Bigr\} \ \vline \ \cdot 2(n-1)  \\
& = \frac{1}{n-1} {n-1\choose k-1}{n-1\choose k} \cdot 2(n-1)
               \ \     \mbox{  (by (\ref{eqn:3.01})) }     \\
& = 2{n-1\choose k-1}{n-1\choose k}.
\end{align*}

We finally return to (\ref{eqn:3.4}) and substitute 
on its right-hand side the values found for the cardinalities 
of $\cC, \cD_{+}$ and $\cD_{-}$, and (\ref{eqn:3.2}) immediately 
follows.
\end{proof}

\begin{center}
\scalebox{.90}{\includegraphics{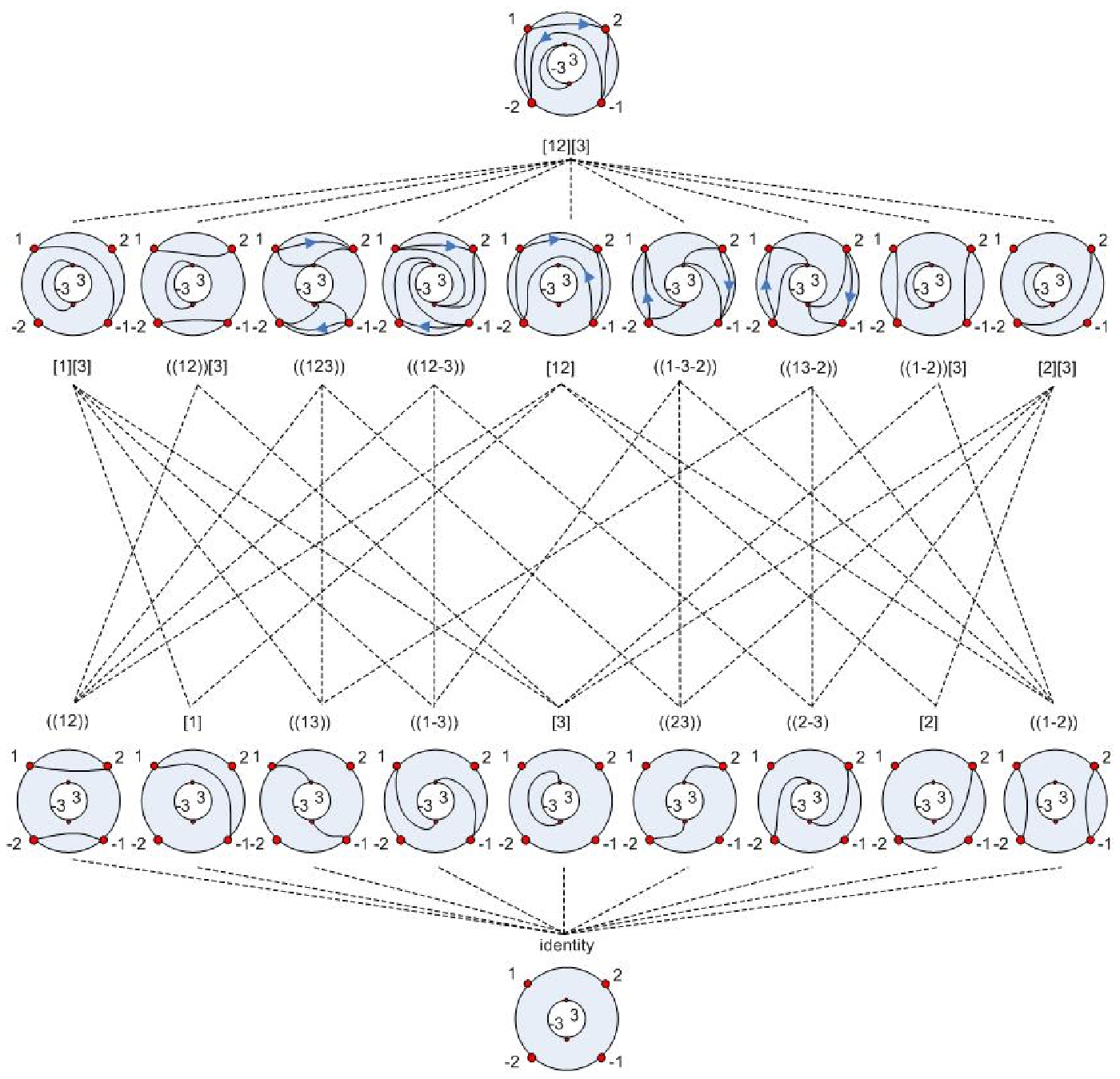}}
\end{center}

\begin{tabular}{cl}
{\bf Figure 2.}
  & The Hasse diagram for $\ncb (2,1)$. The bracket notations
                                      $(( \cdots ))$            \\
  & and $[ \cdots ]$ refer to the cycles of the corresponding
                                        permutations            \\
  & (e.g. $((1,2,-3))$ and $((1,-2))[3]$ are shorthand notations
                                             for the            \\
  & permutations
    $(1,2,-3)(-1,-2,3)$ and $(1,-2)(-1,2)(3,-3)$, respectively).
\end{tabular}

\begin{rem}  \label{rem:3.3}
We note the somewhat surprising fact that $\ncb (n-1,1)$
has exactly the same rank generating function as the lattice
$\ncb (n)$. For $n=2$ we have in fact
$\ncb (1,1) = \ncb (2)$ (equality of sets of partitions of
$\{ 1,2 \} \cup \{ -1, -2 \}$). But already for $n=3$ it is
no longer true that $\ncb (2,1) = \ncb (3)$; moreover, by
comparing the Hasse diagrams of $\ncb (2,1)$ and of $\ncb (3)$,
one easily sees that $\ncb (2,1) \not\simeq \ncb (3)$. (The Hasse
diagram for $\ncb (2,1)$ is drawn in Figure 2 of this paper, while
the one for $\ncb (3)$ appears on page 199 of Reiner's paper
\cite{R97}. In order to establish that
$\ncb (2,1) \not \simeq \ncb (3)$ one can for instance count
edges in the Hasse diagrams -- the Hasse diagram for $\ncb (2,1)$
has 46 edges, while the one for $\ncb (3)$ has 44 edges.)

By comparing the specific formulas that give the
M\"obius functions for $\ncb (n)$ and for $\ncb (n-1,1)$, we will
find in fact that $\ncb (n-1,1) \not\simeq \ncb (n)$ for all
$n \geq 3$; see Remark \ref{rem:3.7} below.

So let us now consider the M\"obius function of $\ncb (n-1,1)$. Its
calculation will be presented in Theorem \ref{thm:3.6}, and will
be based on a partial M\"obius inversion formula that is
described as follows.
\end{rem}

\begin{lemma}  \label{lemma:3.4}
Let $P$ be a finite lattice, let $\widehat{0}$ and
$\widehat{1}$ denote the minimal and the maximal element of $P$,
respectively, and let $\omega$ be a fixed element of $P$, where
$\omega \neq \widehat{1}$. Then
\begin{equation}   \label{eqn:3.7}
\sum\limits_{  \substack{\pi\in P \\
                         \pi \wedge \omega=\widehat{0}} }
\mu_{{}_P}( \pi, \widehat{1} \, ) = 0.
\end{equation}
\end{lemma}

For a proof of Lemma \ref{lemma:3.4}, see Corollary 3.9.3
of \cite{Sta97}. A few other facts needed in the proof of 
Theorem \ref{thm:3.6} are collected in the next remark.

\begin{rem}  \label{rem:3.5}
$1^o$ We will use the explicit formulas known for the M\"obius 
functions of the posets $\nca (n)$ and $\ncb (n)$.

(A) For every $n \geq 1$ we have that 
\begin{equation}  \label{eqn:3.101}
\mu_{\nca (n)} ( \, \widehat{0}, \widehat{1} \, ) = 
(-1)^{n+1}\frac{(2n-2)!}{(n-1)!\,\,n!},
\end{equation}
where $\mu_{\nca (n)}$ is the M\"obius function of $\nca (n)$, and 
$\widehat{0}, \widehat{1}$ are the minimal and 
maximal element of $\nca (n)$, respectively.
(See Theorem~6 of~\cite{K72}.)

(B) For every $n \geq 1$ we have that
\begin{equation}  \label{eqn:3.102}
\mu_{\ncb (n)} ( \, \widehat{0}, \widehat{1} \, ) = 
(-1)^n \cdot \left( \begin{array}{c} 2n-1 \\ n \end{array} \right) ,
\end{equation}
where $\mu_{\ncb (n)}$ is the M\"obius function of $\ncb (n)$, and 
$\widehat{0}, \widehat{1}$ now stand for the minimal and 
maximal element of $\ncb (n)$, respectively.
(See Proposition 7 of \cite{R97}.)

$2^o$ Let $p,q$ be positive integers. It is an easy exercise (left 
to the reader) to check that the formula
\begin{equation}  \label{eqn:3.103}
C ( \tau ) := \tau^{-1} \gamma_{p,q}, \ \ 
\tau \in \sncb (p,q),
\end{equation}
defines a bijection $C: \sncb (p,q) \to \sncb (p,q)$, that is 
order-reversing -- for $\sigma , \tau \in \sncb (p,q)$ one has that
$\sigma \leq \tau \Leftrightarrow C( \sigma ) \geq C( \tau )$,
where the partial order on $\sncb (p,q)$ is as in Definition
\ref{def:2.1}.2. 

Now, by using the canonical isomorphism 
$\otilda : \sncb (p,q) \to \ncb (p,q)$ (see Remark
\ref{rem:2.3}.1), we can transport the map $C$ from (\ref{eqn:3.103})
to an anti-isomorphism $K : \ncb (p,q) \to \ncb (p,q)$, defined
via the formula
\begin{equation}  \label{eqn:3.104}
K ( \, \otilda ( \tau ) \, ) = 
\otilda ( \tau^{-1} \gamma_{p,q} ), \ \ 
\tau \in \sncb (p,q).
\end{equation}
This anti-isomorphism $K$ is the $\ncb (p,q)$--analogue for an
anti-isomorphism of the lattice $\nca (n)$ introduced by 
Kreweras in \cite{K72}, and which is commonly called the 
Kreweras complementation map. Following this trend, we will 
also refer to the map $K$ from (\ref{eqn:3.104}) as
the {\em Kreweras complementation map} of $\ncb (p,q)$.
Note that, due to the fact that it is an anti-isomorphism, the 
Kreweras complementation map has the property that
\begin{equation}  \label{eqn:3.105}
\mu ( \pi , \rho ) = \mu ( \, K( \rho ), K( \pi ) \, ), \ \ 
\forall \, \pi , \rho \in \ncb (p,q) \mbox{ such that } 
\pi \leq \rho ,
\end{equation}
where $\mu$ is the M\"obius function of $\ncb (p,q)$.
\end{rem}

\begin{thm}  \label{thm:3.6}
Let $n \geq 2$ be an integer, let $\mu_{\ncb (n-1,1)}$ be the 
M\"obius function of $\ncb (n-1,1)$, and let 
$\widehat{0}, \, \widehat{1}$ be the minimal and 
maximal element of $\ncb (n-1,1)$, respectively. Then
\begin{equation}  \label{eqn:3.8}
\mu_{\ncb (n-1,1)} ( \,  \widehat{0} , \widehat{1} \, ) =
(-1)^n \cdot \left( \begin{array}{c} 2n-1 \\ n \end{array} \right) 
\cdot\frac{5n-4}{4n-2} .
\end{equation}
\end{thm}

\begin{proof}
Throughout the proof we will write ``$\mu$'' instead 
of ``$\mu_{\ncb (n-1,1)}$'', for compactness.
We will apply Lemma \ref{lemma:3.4} to the
particular case when $P=\ncb (n-1,1)$ and 
\begin{equation}  \label{eqn:3.110}
\omega := \Bigl\{ \ \{ \pm 1, \ldots , \pm (n-1) \}, \ 
                    \{ n \}, \ \{ -n \} \ \Bigr\} .
\end{equation}
By taking into account that the meet operation of $\ncb (n-1,1)$ is 
just the usual ``intersection'' meet, one immediately sees that the 
partitions in the set 
$\{ \, \pi \in \ncb (n-1,1) \mid \pi \wedge \omega$
$= \widehat{0} \, \}$ can be listed explicitly as 
$\widehat{0}, \pi_0 , \pi_1, \ldots , \pi_{n-1} ,
\pi_{-1}, \ldots , \pi_{-(n-1)}$, where 
\[
\pi_0 := \Bigl\{ \ \{ n, -n \} , \ \{ 1 \}, \ \{ -1 \}, \ldots ,
\{ n-1 \}, \ \{ -(n-1) \} \ \Bigr\}
\]
and where for every $i \in \{ \pm 1, \ldots , \pm (n-1) \}$ we put
\[
\pi_i := \Bigl\{ \ \{ i, n \} , \ \{ -i, -n \} \ \Bigr\} \cup
\Bigl\{ \, \{ j \} \ \mid \ j \in \{ \pm 1, \ldots , \pm (n-1) \} , 
\ |j| \neq |i| \Bigr\} .
\]
When applied to this particular situation, Lemma \ref{lemma:3.4}
thus implies that 
\begin{equation}  \label{eqn:3.11}
0 = \mu( \, \widehat{0}, \widehat{1} \, )
+ \mu \bigl( \pi_0 , \widehat{1} \, \bigr)
+\sum_{i=1}^{n-1}   \mu \bigl( \pi_i , \widehat{1} \, \bigr)
+\sum_{i=1}^{n-1}   \mu \bigl( \pi_{-i} , \widehat{1} \, \bigr) .
\end{equation}
It is convenient to consider the equivalent 
restatement of (\ref{eqn:3.11}) that is obtained by 
taking Kreweras complements and by invoking formula
(\ref{eqn:3.105}) from Remark \ref{rem:3.5}.2:
\begin{equation}  \label{eqn:3.12}
0 = \mu( \, \widehat{0}, \widehat{1} \, )
+ \mu \bigl( \, \widehat{0} , \rho_0   \bigr)
+\sum_{i=1}^{n-1}   \mu \bigl( \ \widehat{0} , \rho_i \bigr)
+\sum_{i=1}^{n-1}   \mu \bigl( \ \widehat{0} , \rho_{-i} \bigr) ,
\end{equation}
where we denoted 
$\rho_i := K ( \pi_i ), \mbox{ for } i \in \{ 0 \} \cup
\{ \pm 1 , \ldots , \pm (n-1) \}$. 

Let us now determine explicitly the partitions $\rho_0$ 
and $\rho_{\pm 1}, \ldots , \rho_{\pm (n-1)}$. We do this by using
the corresponding permutations in $\sncb (n-1,1)$, and formula
(\ref{eqn:3.104}) from Remark \ref{rem:3.5}.2. For 
$i \in \{ \pm 1 , \ldots , \pm (n-1) \}$ we write 
$\pi_i = \otilda ( \tau_i )$ with $\tau_i = (i,n)(-i,-n) \in B_n$, 
and we compute 
\begin{align*}
\tau_i^{-1} \gamma_{n-1,1} 
& = \Bigl( \, (i,n) (-i, -n) \, \Bigr) \,
\Bigl( \, (1, \ldots , n-1, -1, \ldots , -(n-1)) (n,-n) \, \Bigr) \\
& = \Bigl( (1, \ldots , i-1,n,-i, \ldots, -(n-1) \Bigr) \,
\Bigl( (-1, \ldots , -(i-1),-n, i, \ldots, n-1 \Bigr) .
\end{align*}
Since $\rho_i = K ( \, \otilda ( \tau_i ) \, )$
= $ \otilda ( \, \tau_i^{-1} \gamma_{n-1,1} \, )$, we thus 
obtain that, for $i>0$,
\begin{equation}  \label{eqn:3.106}
\rho_i = \Bigl\{ \, \{ 1, \ldots , i-1,n,-i, \ldots, -(n-1) \} , \ 
\{ -1, \ldots , -(i-1),-n, i, \ldots, n-1 \} \, \Bigr\} ,
\end{equation}
with a similar formula (left to the reader) in the case $i<0$.
For $\rho_0$ one does a similar calculation, by writing 
$\pi_0 = \otilda ( \tau_0 )$ for $\tau_0 = (n,-n) \in B_n$. The 
reader should have no difficulty in checking that this calculation 
simply leads to the equality $\rho_0 = \omega$, with $\omega$ 
taken from (\ref{eqn:3.110}).

{}From the explicit form found in (\ref{eqn:3.106}) for $\rho_i$
with $i \in \{ \pm 1 , \ldots , \pm (n-1) \}$, one easily infers 
that the interval $[ \, \widehat{0}, \rho_i ]$ of $\ncb (n-1,1)$ 
is poset isomorphic with the lattice $NC^{(A)} (n)$. Indeed, the 
process of constructing a partition $\sigma \in \ncb (n-1,1)$ such 
that $\sigma \leq \rho_i$ amounts precisely to breaking in a 
non-crossing way the block 
$\{ 1, \ldots , i-1,n,-i, \ldots, -(n-1) \}$ of $\rho_i$,
where the cyclic order of the $n$ elements of the block is as 
listed above. (This must be of course matched by the corresponding,
uniquely determined, non-crossing breaking of the other block
$\{ -1, \ldots , -(i-1),-n, i, \ldots, n-1 \}$ of $\rho_i$.) 
The isomorphism
$[ \, \widehat{0}, \rho_i ] \simeq \nca (n)$ and 
(\ref{eqn:3.101}) thus give us that
\[
\mu( \, \widehat{0},\rho_i)=(-1)^{n+1}\frac{(2n-2)!}{(n-1)!\,\,n!} .
\]
In a similar way, one finds that the interval 
$[ \, \widehat{0}, \rho_0 ]$ of $\ncb (n-1,1)$ is isomorphic with 
$\ncb (n-1)$, and hence (by (\ref{eqn:3.102}))
we have
\[
\mu\bigl( \, \widehat{0}, \rho_0\bigr)
=(-1)^{n-1}{2n-3\choose n-1} .
\]

Finally, by substituting in (\ref{eqn:3.12}) the concrete 
values obtained above for the $\mu ( \, \widehat{0}, \rho_i )$, 
we find that 
\[
- \mu( \, \widehat{0}, \widehat{1} \, )
= (-1)^{n-1}{2n-3\choose n-1}
+ (2n-2)\cdot(-1)^{n-1}\cdot\frac{(2n-2)!}{(n-1)!\,\,n!},
\]
and the required formula for $\mu( \, \widehat{0}, \widehat{1} \, )$
follows by straightforward calculation.
\end{proof}

\begin{rem}  \label{rem:3.7}
By comparing the formula (\ref{eqn:3.8}) found in Theorem 
\ref{thm:3.6} against the corresponding formula (\ref{eqn:3.102}) 
that holds for $\ncb (n)$, we see that 
$\mu_{\ncb (n)} ( \, \widehat{0} , \widehat{1} \, )$ is different
from $\mu_{\ncb (n-1,1)} ( \, \widehat{0} , \widehat{1} \, )$ for 
all $n \geq 3$. This implies, of course, that 
$\ncb (n-1,1) \not\simeq \ncb (n)$ for $n \geq 3$.
\end{rem}

\section{Rank generating function for \boldmath{$\ncb (p,q)$} }
\setcounter{section}{4}

\setcounter{equation}{0}

In this section, we determine the rank generating function 
for $NC^{(B)}(p,q)$. Our results follow directly from a bijection, 
in Proposition~\ref{prop:4.2} below, which is
similar to Lemma 2.1 of~\cite{E80} and Proposition 6 of~\cite{R97}. 
As a preliminary, we have the following discussion of strings of 
parentheses.

\begin{rem}\label{cycle}
We let $\{ (,)\}^*$ be the set of strings of left parentheses 
``$($'' and right parentheses ``$)$''. With multiplication given by 
concatenation, this set forms a monoid, with the empty string acting
as identity element.

If $s=s_1\ldots s_n\in\{ (,)\}^*$, $n\ge 1$, then the nontrivial 
left-substrings of $s$ are given by $u_i := s_1\ldots s_i$,  
$i=1,\ldots ,n$. If all nontrivial left-substrings of $s$ have 
(strictly) more left parentheses than right parentheses, then we 
will say that $s$ is {\em legal from the left}. 

For $s=s_1\ldots s_n\in\{ (,)\}^*$, $n\ge 1$, the {\em cyclic 
shifts} of $s$ are the $n$ strings 
\[
s^{(1)} = s_2 \cdots s_n s_1, \ 
s^{(2)} = s_3 \cdots s_n s_1 s_2, \ \ldots , \ 
s^{(n-1)} = s_n s_1 s_2 \cdots s_{n-1}, \ s^{(n)} = s.
\]
Suppose that $s$ has $m$ more left parentheses than right parentheses, 
for some $m\ge 1$. Then the well-known Cycle Lemma (see for instance
the discussion on page 67 of~\cite{Sta99}) says that exactly $m$ of 
the cyclic shifts of $s$ are legal from the left. 

For example, if $s$ is the string 
$( \, ) \, ( \, ( \, ) \, ( \, ( \ $, which has $5$ left 
parentheses and $2$ right parentheses, then the $3$ cyclic shifts 
of $s$ that are legal from the left are 
\[
s^{(2)}=( \, ( \, ) \, ( \, ( \, ( \, ),  \ \
s^{(5)}=( \, ( \, ( \, ) \, ( \, ( \, ),  \ \
s^{(6)}=( \, ( \, ) \, ( \, ( \, ) \, ( \ .
\]

Symmetrically, if all nontrivial right-substrings of $s$ have more 
right parentheses than left parentheses, then we say that $s$ 
is {\em legal from the right}. For this case, suppose that $s$ has 
$m$ more right parentheses than left parentheses, for some $m\ge 1$.
Then the Cycle Lemma says that exactly $m$ of the cyclic shifts of 
$s$ are legal from the right. 
\end{rem}

\begin{prop} \label{prop:4.2}
Let $p,q$ be positive integers. Suppose that $c,e,i$ are 
integers satisfying the inequalities stated in $(\ref{eqn:2.10})$ 
of Remark \ref{rem:2.4}, that is: $1 \leq c \leq \min \{ p, q \}$ and 
$0 \leq e \leq p-c$, $0 \leq i \leq q-c$. Then there exists a 
bijection between the set
\begin{equation}  \label{eqn:4.101}
\left\{ (d,L^E,R^E,L^I,R^I)  \begin{array}{cl}
\vline & 1 \leq d \leq 2c                               \\
\vline & L^E,R^E \subseteq \{ 1, \ldots , p \} , \  
              \vert L^E\vert =e+c, \ \vert R^E\vert =e, \\
\vline & L^I,R^I \subseteq \{ p+1, \ldots , p+q \} , \
              \vert L^I\vert =i, \ \vert R^I\vert =i+c
\end{array}  \right\}
\end{equation}
and the set of partitions in $NC^{(B)}(p,q)$ that have 
connectivity equal to $c$, have $e$ exterior pairs of blocks, 
and have $i$ interior pairs of blocks.
\end{prop}

\begin{proof}
We will describe explicitly the constructions for two maps 
$(d, L^E, R^E, L^I, R^I) \mapsto \pi$ and 
$\pi \mapsto (d, L^E, R^E, L^I, R^I)$, and we will leave it 
as an exercise to the reader to check that these two maps are 
inverse to each other (thus giving together a bijection as stated).
We recommend that the general descriptions given below for the 
two maps are read in parallel with Remark \ref{rem:4.25}, which
illustrates how the maps work on a concrete example.

\vspace{10pt}

{\em A. Description of the map $(d, L^E, R^E, L^I, R^I) \mapsto \pi$.}
Given $(d,L^E,R^E,L^I,R^I)$ as in (\ref{eqn:4.101}), insert left 
and right parentheses into the string 
$$1,\ldots ,p,-1,\ldots ,-p$$
by placing a left (respectively right) parenthesis before (respectively 
after) each occurrence of $j$ and $-j$, for each value $j$ in $L^E$ 
(respectively $R^E$). In this way we obtain the string $u$ of length 
$2(p+2e+c)$, consisting of numbers and parentheses. In $u$, there are
$2c$ more left parentheses than right parentheses, so the Cycle Lemma 
in Remark~\ref{cycle} implies that there are $2c$ cyclic shifts of $u$ 
beginning with a left parenthesis such that the subsequence consisting 
of parentheses only is legal from the left. Suppose that these $2c$ 
cyclic shifts are given by $u^{(i_1)}, \ldots ,u^{(i_{2c})}$, ordered 
so that $i_1<\cdots <i_{2c}$. Then let $t_1=u^{(i_d)}$.

Similarly, insert left and right parentheses into the string
$$p+1,\ldots ,p+q,-(p+1), \ldots ,-(p+q)$$
by placing a left (respectively right) parenthesis before (respectively
after) each occurrence of $j$ and $-j$, for each value $j$ in $L^I$ 
(respectively $R^I$), to obtain the string $v$ of numbers and 
parentheses. In $v$ there are $2c$ more right parentheses than left 
parentheses, so the Cycle Lemma in Remark~\ref{cycle} implies that 
there are $2c$ cyclic shifts of $v$ ending with a right parenthesis
such that the subsequence consisting of parentheses only is legal 
from the right. If these $2c$ cyclic shifts are given by 
$v^{(j_1)}, \ldots ,v^{(j_{2c})}$, ordered
so that $j_1<\cdots <j_{2c}$, then let $t_2=v^{(j_{2c})}$.

Now consider the concatenation $t_1 t_2$ of the two strings $t_1$ 
and $t_2$ found above. From the string $t_1t_2$ we read off a unique 
partition $\pi$ in $NC^{(B)}(p,q)$ in the following way: the numbers 
inside a lowest-level pair of matching parentheses form a block of 
$\pi$; remove these numbers and this pair of parentheses from the
string, and iterate until the string is empty. (See part A of 
Remark \ref{rem:4.25} below, for a concrete example of how this 
works.)

\vspace{10pt}

{\em B. Description of the map $\pi \mapsto (d, L^E, R^E, L^I, R^I)$.}
Let $\pi$ be a partition in $\ncb (p,q)$ that has connectivity equal 
to $c$, has $e$ pairs of external blocks and has $i$ pairs of 
internal blocks. A significant fact we we will use here is 
that (even though $\pi$ is drawn on a circular picture) every 
block of $\pi$ that is either an external block or an internal
block comes with a canonical total order on it, and thus has a
{\em first element} and a {\em last element}.

Indeed, say for instance that $A$ is an external (\textit{i.e.}, such that
$A \subseteq \{ \pm 1, \ldots , \pm p \}$) block of $\pi$. Let 
us choose an element $i \in -A$ and, by starting from this $i$,
let us travel around the external circle of the annulus (in the 
sense that we always use for this circle -- that is, 
clockwise). When we do this, we encounter the elements of $A$
in a certain order, and this order does not depend on our choice 
of the starting point $i \in -A$. (The latter fact is an 
immediate consequence of the fact that the blocks $A$ and $-A$
of $\pi$ do not cross.) We thus end with a ``canonical'' total 
order for the elements of $A$. Clearly, a similar argument can 
be given when we deal with an internal block of $\pi$. 
Moreover, the same argument can be also used to introduce 
a total order on each of the sets 
$A \cap \{ \pm 1, \ldots , \pm p \}$ and 
$A \cap \{ \pm (p+1), \ldots , \pm (p+q) \}$, in the case 
when $A$ is a connecting block of $\pi$.

So then, starting from the given $\pi \in \ncb (p,q)$, let us
draw some parentheses on the picture representing $\pi$, according 
to the following recipe:

\noindent
(a) For every block $A$ of $\pi$ that is either an external block 
or an internal block, we draw a left parenthesis immediately before 
the first element of $A$, and a right parenthesis immediately after 
the last element of $A$.

\noindent
(b) For every connecting block $A$ of $\pi$ we draw a left 
parenthesis immediately before the first element of 
$A \cap \{ \pm 1, \ldots , \pm p \}$, and a right parenthesis 
immediately after the last element of 
$A \cap \{ \pm (p+1), \ldots , \pm (p+q) \}$.

But now, if the parentheses added to the picture of $\pi$ are 
read starting from 1 on the outside circle and starting from 
$p+1$ on the inside circle, then one gets two strings of numbers 
and parentheses $u$ and $v$, that are exactly of the same kind 
as the strings denoted by ``$u$'' and ``$v$'' in 
part A of the proof. Furthermore, it is immediate that the strings 
$u$ and $v$ obtained here correspond to some subsets 
\[
L^E , R^E \subseteq \{ 1, \ldots , p \} , \qquad  
L^I , R^I \subseteq \{ p+1, \ldots , p+q \} , 
\]
which have the properties required in (\ref{eqn:4.101}). 

In order to complete the description of the map 
$\pi \mapsto (d, L^E, R^E, L^I, R^I)$, we are thus left 
to explain how we obtain the number $d \in \{ 1, \ldots , 2c \}$.
It is immediate that determining $d$ (in the context where
we have already identified the strings $u$ and $v$) is equivalent 
to choosing one of the $2c$ cyclic shifts of $u$ that are 
legal from the left. Expressed directly in terms of the partition 
$\pi$, this in turn amounts to choosing one of the $2c$ connecting 
blocks of $\pi$. (To be precise: if a connecting block $A$ is 
chosen, then we pick the cyclic shift of $u$ that starts with 
the left parenthesis placed immediately before the first element of 
$A \cap \{ \pm 1, \ldots , \pm p \}$.) So what we have to do is to
give a procedure for canonically selecting a connecting 
block of $\pi$. To do so, we look at how the 
connecting blocks intersect the interior circle of the annulus: 
we start from $p+1$ and move counterclockwise around the 
interior circle, and stop the first time that we meet an element 
belonging to a connecting block. (See part B of Remark 
\ref{rem:4.25} below for a concrete example of how this works.)
\end{proof}

\begin{rem}  \label{rem:4.25}
Let us illustrate how the two maps described in the proof of the 
preceding proposition work on a concrete example. Consider the 
situation when the integers $p,q,c,e,i$ given in Proposition
\ref{prop:4.2} are $p=5$, $q=3$, $c=1$, $e=2$, $i=1$. 

\vspace{10pt}

{\em A.} Let us determine explicitly the partition 
$\pi \in \ncb (5,3)$ that corresponds (by the first of the two 
maps described in the proof of Proposition \ref{prop:4.2}) to 
the tuple $(d, L^E, R^E, L^I, R^I)$ where 
\begin{equation}  \label{eqn:4.215}
d=2, \ L^E=\{ 2,4,5\}, \ R^E=\{ 1,2\}, \ L^I=\{ 7\}, \ R^I=\{ 6,7\} .
\end{equation}
By inserting parentheses in $1, \ldots , 5,-1, \ldots ,-5$ we 
obtain the following string of length 20, consisting of numbers and 
parentheses:
\begin{equation}  \label{eqn:4.202}
u=1 \, ) \, ( \, 2 \, ) \, 3 \, ( \, 4 \, ( \, 5 \, -1 \, ) 
\, ( \, -2 \, ) \, -3 \, ( \, -4 \, ( \, -5.
\end{equation}
The two cyclic shifts of $u$ that begin with a left parenthesis and
are legal from the left are $u^{(6)}$ and $u^{(16)}$. Since we have 
$d=2$, the string $t_1$ from the description of the above bijection 
is hence:
\[
t_1=u^{(16)}= ( \, -4 \, ( \, -5 \, 1 \, ) \, ( \, 2 \, ) 
\, 3 \, ( \, 4 \, ( \, 5 \,-1 \, ) \, ( \, -2 \, ) \, -3.
\]
In a similar way, by inserting parentheses in $6,7,8,-6,-7,-8$ 
we get
\begin{equation}  \label{eqn:4.203}
v=6 \, ) \, ( \, 7 \, ) \, 8 \ -6 \, ) \, ( \, -7 \, ) \, -8  ,
\end{equation}
and then have
\[
t_2 = v^{(8)} = 
( \, -7 \, ) \, -8 \ 6 \, ) \, ( \, 7 \, ) \, 8 \, -6 \ ).
\]
Finally, we concatenate $t_1$ and $t_2$, and from the string $t_1t_2$ 
we read off the desired partition $\pi \in \ncb (5,3)$, which is
\begin{equation}  \label{eqn:4.204}
\pi = \Bigl\{ \, \{ 1,-5\},\{-1,5\},\{2\},\{-2\},
\{3,-4,-6,8\},\{-3,4,6,-8\},\{7\},\{-7\} \, \Bigr\} .
\end{equation}

\vspace{10pt}

{\em B.} Conversely, let us now start from the partition
$\pi \in \ncb (5,3)$ that appeared in (\ref{eqn:4.204}) above,
and determine explicitly the tuple 
$(d, L^E, R^E, L^I, R^I)$ that corresponds to $\pi$ by the 
second map described in the proof of Proposition \ref{prop:4.2}.

The annular picture for $\pi$ and the parentheses that have to 
be added to it are shown in Figure 3 below. When looking at 
Figure 3, the reader should keep in mind that the placing of a 
parenthesis ``immediately 
before'' (or ``imediately after'') a labelled point on one of 
the circles of the annulus must be always done in agreement 
with the chosen running direction on that particular circle. 
Thus for instance the parenthesis sitting next to $-5$ in Figure 3
is a ``left parenthesis placed immediately before $-5$'', since 
the outside circle is run clockwise; while next to 6 we have a
``right parenthesis placed immediately after 6'', as the running 
direction on the inside circle is counterclockwise.

\begin{center}
\scalebox{.48}{\includegraphics{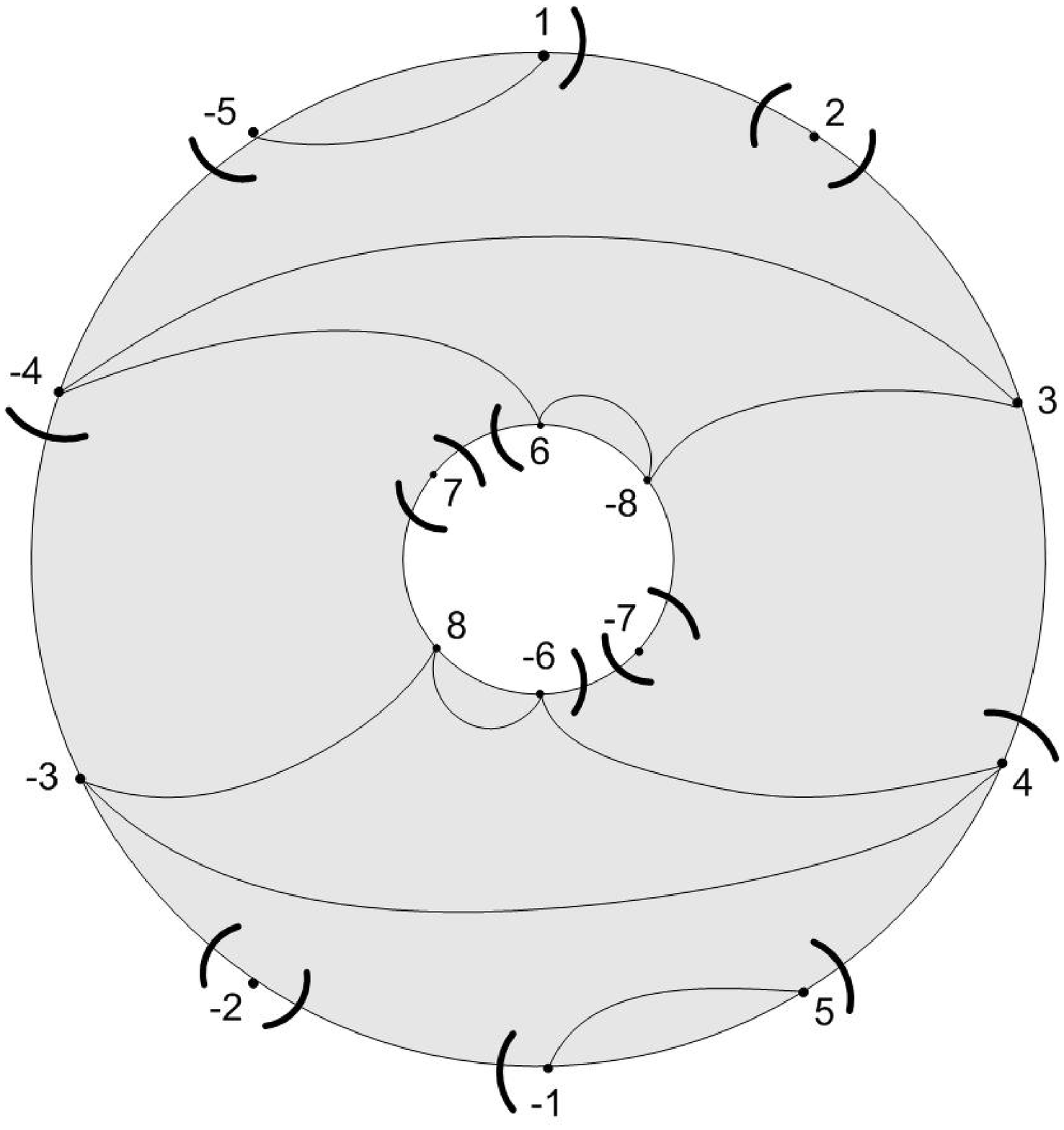}}

{\bf Figure 3.} Adding parentheses to the picture of a 
partition in $\ncb (p,q)$.
\end{center}

$\ $

If we read Figure 3 starting with 1 on the 
outside circle and starting with 6 on the inside circle, we find 
the strings $u$ and $v$ displayed in (\ref{eqn:4.202}) and 
(\ref{eqn:4.203}), respectively, and from these $u$ and $v$ we 
clearly get back to the sets $L^E, R^E$, $L^I, R^I$ indicated 
in (\ref{eqn:4.215}). 

Finally, let us also follow on Figure 3 the procedure for finding 
the value of $d$. What we have to do is to start from $p+1 (=6)$ and 
move counterclockwise around the interior circle of the annulus, 
and stop the first time that we meet an element belonging to a 
connecting block. But in this example the number $6$ belongs 
to the connecting block $A = \{ 3,-4,6,-8 \}$ of $\pi$; 
so this is the connecting block of $\pi$ that is chosen. The 
first element of $A \cap \{ \pm 1, \ldots , \pm p \}$ is $-4$, 
hence we choose the cyclic shift of $u$ that starts with 
``$( \ -4$'', and this corresponds to the value $d=2$.
\end{rem}

\begin{cor}  \label{cor:4.3}
Let $p,q,c,e,i$ be integers such that 
$1 \leq c \leq \min \{ p,q \}$ and 
$0 \leq e \leq p-c$, $0 \leq i \leq q-c$. Then there are exactly
\begin{equation}  \label{eqn:4.301}
2c \, {p\choose e} {p\choose e+c} 
\ {q\choose i} {q\choose i+c} 
\end{equation}
partitions in $NC^{(B)}(p,q)$ that have connectivity equal to $c$, 
have $e$ exterior pairs of blocks, and have $i$ interior pairs of 
blocks.
\end{cor}

\begin{proof}
This follows by taking cardinalities in the bijection from 
Proposition \ref{prop:4.2}.
\end{proof}

As an immediate consequence of the above corollary, one can 
enumerate the partitions in $\ncb (p,q)$ by their connectivity.

\begin{thm}  \label{thm:4.4}
Let $p,q$ be positive integers.

$1^o$ For every $1 \leq c \leq \min \{ p,q \}$, there are exactly
\begin{equation}  \label{eqn:4.401}
2c{2p\choose p-c}{2q\choose q-c}
\end{equation}
partitions in $\ncb (p,q)$ that have connectivity equal to $c$.

$2^o$ There are exactly
\begin{equation}  \label{eqn:4.402}
{2p\choose p}{2q\choose q}
\end{equation}
partitions in $\ncb (p,q)$ that have connectivity equal to $0$.

$3^o$ The total number of partitions in $\ncb (p,q)$ is 
\begin{equation}  \label{eqn:4.403}
\left|NC^{(B)}(p,q)\right| =
\frac{p+q+pq}{p+q} \cdot {2p\choose p}{2q\choose q}.
\end{equation}
\end{thm}

\begin{proof}
$1^o$ From Proposition~\ref{prop:4.2}, the number of partitions of 
connectivity $c$ in $\ncb (p,q)$ equals
\begin{align*}
2c \sum_{e,i\ge 0} {p \choose e}{p\choose e+c} 
                 \ {q\choose i}{q\choose i+c}                   
&= 2c \, \Bigl( \, \sum_{e=0}^{p-c} 
                {p \choose e}{p\choose e+c} \, \Bigr)  
  \,   \Bigl( \, \sum_{i=0}^{q-c} 
                {q \choose i}{q\choose i+c} \, \Bigr) \\
&= 2c {2p\choose p-c}{2q\choose q-c} 
\end{align*}
(where for the second equality we used the identity 
(\ref{vanderm})). 

$2^o$ As observed in Remark \ref{rem:2.5}, the partitions with 
connectivity 0 in $\ncb (p,q)$ are given by the  direct product
of $NC^{(B)}(p)$ with $NC^{(B)}(q)$; hence their number is
\[
\left|NC^{(B)}(p)\right| \cdot \left|NC^{(B)}(q)\right| =
{2p\choose p} \cdot {2q\choose q} \qquad
\ \mbox{ (by using (\ref{eqn:3.021})).}
\]

$3^o$ From the above it follows that 
\begin{equation}  \label{eqn:4.404}
\left|NC^{(B)}(p,q)\right| = {2p\choose p}{2q\choose q}
+\sum_{c\ge 1} 2c{2p\choose p-c}{2q\choose q-c}.
\end{equation}
In the summation over $c$ in~(\ref{eqn:4.404}), we observe that 
the ratio of two consecutive terms is a rational function of $c$,
hence we are dealing with a hypergeometric summation. Referring to 
the standard notations for hypergeometric series one sees, more
precisely, that 
\begin{equation}  \label{eqn:4.405}
\sum_{c\ge 1} 2c \, {2p\choose p-c}{2q\choose q-c}
= 2 \, {2p \choose p-1}{2q \choose q-1} \cdot
{ }_3F_2 \left(  \begin{array}{c}
2 , \ -(p-1), \ -(q-1)  \\
p+2, \ q+2 
\end{array}  ; 1 \right) .
\end{equation}
(For the precise definition of ${ }_3F_2$ see for instance 
formula (2.1.2) on page 62 of \cite{AAR99}.)

Now, it turns out that the special ${ }_3F_2$ series on the 
right-hand side of (\ref{eqn:4.405}) can be summed in closed form; 
this is by a theorem of Dixon (see formula (2.2.11) on page 72 of 
\cite{AAR99}), which gives us that
\begin{equation}  \label{eqn:4.406}
{ }_3F_2 \left(  \begin{array}{c}
2 , \ -(p-1), \ -(q-1)  \\
p+2, \ q+2 
\end{array}  ; 1 \right) \ = \ 
\frac{(p+1)(q+1)}{2(p+q)} .
\end{equation}
By substituting this expression into (\ref{eqn:4.405}), and 
then by plugging the result into (\ref{eqn:4.404}), we obtain the 
stated formula for the cardinality of $\ncb (p,q)$.
\end{proof}

{}From Corollary \ref{cor:4.3} one can also infer a formula for the
rank generating function of $\ncb (p,q)$.

\begin{thm}  \label{thm:4.5}
Let $p,q$ be positive integers and let $F(x)$ denote the rank 
generating function for $\ncb (p,q)$. Then
\begin{equation*}
F(x) = \sum_{i,j \geq 0} {p\choose i}^2{q\choose j}^2 x^{i+j} 
\ + \ \sum_{c \geq 1} \ \sum_{e,i\ge 0} \
2c \, {p\choose e}{p\choose e+c} 
\, {q\choose i}{q\choose i+c} x^{p+q-e-i-c}.
\end{equation*}
\end{thm}

\begin{proof}
The first summation on the right-hand side of the result
gives the contribution to $F(x)$ from partitions $\pi \in \ncb (p,q)$
that have connectivity equal to $0$. Indeed, we saw in 
Remark \ref{rem:2.5} how such a partition $\pi$ is obtained by 
putting together a partition $\pi_{ext} \in \ncb (p)$ and a 
partition $\pi_{int} \in \ncb (q)$; it is moreover immediate that
when this is done, the rank of $\pi$ in $\ncb (p,q)$ is the sum of 
the ranks of $\pi_{ext}$ and $\pi_{int}$ in $\ncb (p)$ and in 
$\ncb (q)$, respectively. Thus when summing over partitions 
$\pi \in \ncb (p,q)$ with connectivity equal to 0 we get
\begin{align*} 
\sum_{\pi} x^{\mathrm{rank} ( \pi )}
& = \Bigl( \ \sum_{\pi_{ext} \in \ncb (p)} \ 
x^{\mathrm{rank} ( \pi_{ext} )} \ \Bigr) \
\Bigl( \ \sum_{\pi_{int} \in \ncb (q)} \ 
x^{\mathrm{rank} ( \pi_{int} )} \ \Bigr)            \\
& = \Bigl( \ \sum_{i=0}^p 
{p\choose i}^2 x^i \ \Bigr) \
\Bigl( \ \sum_{j=0}^q 
{q\choose j}^2 x^j \ \Bigr)  \qquad \mbox{  (by (\ref{eqn:3.02})).}
\end{align*}

On the other hand, let us observe that if $\pi \in \ncb (p,q)$
has connectivity $c \geq 1$, has $e$ pairs of exterior blocks and
has $i$ pairs of internal blocks, then from (\ref{eqn:2.7}) it
follows that
\[
\mbox{rank} ( \pi ) = (p+q) - (c+e+i).
\]
Hence in view of Corollary \ref{cor:4.3}, the contribution to 
$F(x)$ from the partitions $\pi \in \ncb (p,q)$
that have connectivity different from 0 is given precisely by
the second summation on the right-hand side of the result.
\end{proof}

\begin{rem}  \label{rem:4.6}
It can be shown that the second summation on the right-hand side of 
Theorem~\ref{thm:4.5} can be reexpressed with only two summation indices
instead of three, in the form:
\begin{equation}  \label{eqn:4.601}
\frac{2pq}{p+q}\sum_{i,j\ge 1}
\left({p\choose i}{q\choose j-1}+{p\choose i-1}{q\choose j}\right)
{p-1\choose i-1}{q-1\choose j-1}x^{i+j-1}.
\end{equation} 
The proof of this fact is technical, and is omitted.
\end{rem}

\section{Zeta polynomial and M\"obius function for 
\boldmath{$\ncb (p,q)$} }
\setcounter{section}{5}

\setcounter{equation}{0}

In this section, we determine the zeta polynomial and M\"obius function 
for $NC^{(B)}(p,q)$. These follow immediately by extending the 
bijection given in Proposition~\ref{prop:4.2} to count multichains in 
$NC^{(B)}(p,q)$, similar to Theorem~3.2 of~\cite{E80} and 
Proposition~7 of~\cite{R97}.

\begin{prop} \label{modbij}
For $p,q\ge 1$ and $m\ge 2$, the bijection given in 
Proposition~\ref{prop:4.2} extends to a bijection between
\begin{equation}  \label{eqn:5.101}
\left\{ 
\begin{array}{rclcl}
                &     &            & \vline   & 1\le c,\quad 1 \leq d \leq 2c   \\
                &     &            & \vline   &                    \\
(c,d;L^E, R_1^E, & \ldots, & R_{m-1}^E;     & \vline   & L^E, R_1^E, 
           \ldots, R_{m-1}^E \subseteq \{ 1, \ldots , p \}         \\
        & L^I, & R_1^I, \ldots , R_{m-1}^I)  &
\vline  & \ \ \mbox{  } \vert L^E \vert = \vert R_1^E \vert 
          + \cdots + \vert R_{m-1}^E \vert + c                     \\
                &     &            & \vline   &                    \\
   &   &   &  \vline  & L^I, R_1^I, 
       \ldots, R_{m-1}^I \subseteq \{ p+1, \ldots , p+q \}         \\
   &   &   & \vline  & \ \ \mbox{  } \vert L^I \vert 
    = \vert R_1^I \vert + \cdots + \vert R_{m-1}^I \vert - c  
\end{array}  \right\}
\end{equation}
and the set of multichains 
$\pi_1 \le \cdots \le \pi_{m-1}$ in $\ncb (p,q)$, in which 
\begin{equation}  \label{eqn:5.102}
\mathrm{rank}(\pi_i)= p+q- \left( \vert R_i^E\vert+ \cdots 
+ \vert R_{m-1}^E\vert + \vert R_i^I\vert
+ \cdots + \vert R_{m-1}^I \vert \right) , \ \ 
1 \leq i \leq m-1,
\end{equation}
and at least one of the $\pi_i$ has positive connectivity.
\end{prop}

\begin{proof}
This proof is to a good extent a repetition of the one 
given earlier for Proposition \ref{prop:4.2} (which corresponds 
to the case $m=2$ of the present proposition). Because of this, 
we will only give an outline of the argument, with emphasis on 
the points that are specific to the situation at hand.

Given a tuple 
$(c,d;L^E,R_1^E, \ldots, R_{m-1}^E; L^I,R_1^I, \ldots, R_{m-1}^I)$
as in (\ref{eqn:5.101}), insert left and right parentheses into 
the string
$$1,\ldots ,p,-1,\ldots ,-p,$$
with $m-1$ types of right parentheses $)^k$ for $k=1,\ldots, m-1$,
as follows: place a left parenthesis before each occurrence of $j$ 
and $-j$, for each value $j$ in $L^E$; for $k=1,\ldots ,m-1$, place 
a right parenthesis of type $)^k$ after each occurrence of $j$ 
and $-j$, for each value $j$ in $R^E_k$. (If $j$ occurs in both 
$R^E_a$ and $R^E_b$, for $a<b$,
then place the corresponding $)^b$ to the right of the $)^a$.)
In the resulting string of numbers and parentheses there are $2c$ 
more left parentheses than right parentheses, so the Cycle Lemma 
in Remark~\ref{cycle} implies that there are $2c$ cyclic shifts of 
the string beginning with a left parenthesis such that the 
subsequence consisting of parentheses only is legal from the left.
Order these $2c$ cyclic shifts in the canonical way (by the same 
method as in the proof of Proposition \ref{prop:4.2}), and choose 
the $d$th of them to give the string $t_1$.

Similarly, insert left and right parentheses into the string
$$p+1,\ldots ,p+q,-(p+1), \ldots ,-(p+q)$$
by placing a left parenthesis before each occurrence of $j$ and $-j$,
for each value $j$ in $L^I$; for $k=1,\ldots ,m-1$, place a right 
parenthesis of type $)^k$ after each occurrence of $j$ and $-j$, for
each value $j$ in $R^I_k$. In the resulting string of numbers and 
parentheses there are $2c$ more right parentheses than left 
parentheses, so the Cycle Lemma in Remark~\ref{cycle}
implies that there are $2c$ cyclic shifts of the sequence ending 
with a right parenthesis such that the subsequence consisting of 
parentheses  only is legal from the right. Let $t_2$ be
the canonical choice (found in the same way as in the proof of 
Proposition \ref{prop:4.2}) from among these $2c$ cyclic shifts.

Now, from the string $t_1t_2$, we create a partition $\pi_1$ in 
$NC^{(B)}(p,q)$ in the following way: the numbers inside a 
lowest-level pair of matching parentheses form a block of $\pi_1$; 
remove these numbers and this pair of parentheses from the
string, and iterate until the string is empty. Then for 
$j=2,\ldots ,m-1$, remove the right parentheses of types 
$)^1, \dots ,)^{j-1}$ from  $t_1t_2$, together with the left 
parentheses that pair with them, and read the remaining string as 
above to obtain $\pi_j$. This produces the multichain 
$\pi_1\le \ldots\le\pi_{m-1}$ in $\ncb (p,q)$, and gives a 
bijection with the required properties.

To see that at least one of
the $\pi_i$ has positive connectivity, note that the string $t_1t_2$ starts
with $(a\ldots$ and ends with $\ldots b)^{\ell_1}\cdots )^{\ell_k}$, where
 $a\in\{1,\ldots ,p,-1,\ldots ,-p\}$, $b\in\{p+1,\ldots ,p+q,-p-1,\ldots ,-p-q\}$,
 $k\ge 1$, and $1\le\ell_1<\cdots <\ell_k\le m-1$. Then, in $t_1t_2$, the
initial left parenthesis $($ is paired
with the terminal right parenthesis $)^{\ell_k}$, and therefore $\pi_i$ must
have positive connectivity when $\ell_{k-1}<i \le\ell_k$, since for these
choices of $i$, elements $a$ and $b$ must appear in the same block of $\pi_i$ in 
the construction above.
\end{proof}

\begin{rem}\label{egmodbij}
As a concrete example for how Proposition~\ref{modbij} works, 
suppose we have $p=6,q=3,m=3$, with $c=2$, $d=1$, 
$L^E=\{ 1,2,3,5,6\}$, $R_1^E=\{ 1,3\}$, $R_2^E=\{ 3\}$, 
and $L^I = \{ 8,9 \}$, $R_1^I=\{ 7,8,9\}$, $R_2^I=\{ 7\}$.
By inserting parentheses in $1, \ldots , 6, -1, \ldots , -6$ we 
obtain the string
$$(1)^1(2(3)^1)^24(5 (6(-1)^1(-2(-3)^1)^2-4(-5(-6 ,$$
which has $4$ cyclic shifts that we might consider.
Since we are given that $d=1$, the cyclic shift that we select
is the one that begins with ``$(5$'', thus getting
$$t_1= (5 (6(-1)^1(-2(-3)^1)^2-4(-5(-6(1)^1(2(3)^1)^24.$$
Similarly, we obtain
$$t_2 =( -8)^1( -9)^1 7)^1)^2(8)^1(9)^1 -7)^1)^2,$$
and from the string $t_1t_2$, we obtain the partitions
$$\pi_1= \Bigl\{ \, \{ 4,-6,7 \}, \{ -4,6,-7 \} \, \Bigr\} \cup 
\Bigl\{ \ \{ i \} \ \vline \ 1 \leq |i| \leq 9, \
|i| \neq 4,6,7 \Bigr\} , $$
$$\pi_2 = \Bigl\{ \, \{ 1,4,-5,-6,7,-8,-9 \}, \{ -1,-4,5,6,-7,8,9 \},
                     \{ 2,3\}, \{ -2,-3\} \, \Bigr\}.$$
Note that $\pi_1\le\pi_2$, and that both $\pi_1$ and $\pi_2$ have positive
connectivity (both have conncetivity equal to $1$).
\end{rem}

As an immediate enumerative consequence of Proposition~\ref{modbij}, 
we obtain the zeta polynomial for $NC^{(B)}(p,q)$.

\begin{thm} \label{thm:5.4}
Let $p,q$ be positive integers.

$1^o$ The zeta polynomial of $\ncb (p,q)$ is given by the 
formula:
\begin{equation}  \label{eqn:5.401}
Z_{\ncb (p,q)} (m)= {mp\choose p}{mq\choose q}+
\sum_{c=1}^p 2c{mp\choose p-c}{mq\choose q+c}.
\end{equation}

$2^o$ The number of maximal chains in $NC^{(B)}(p,q)$ is 
equal to
\begin{equation}  \label{eqn:5.402}
{p+q\choose p} p^p q^q
+\sum_{c\ge 1} 2c{p+q\choose p-c}p^{p-c} q^{q+c}.
\end{equation}
\end{thm}

\begin{proof}
$1^o$ The zeta polynomial $Z_P$ of a partially ordered set $P$
is defined via the condition that for every $m \geq 2$, the 
value $Z_P (m)$ is equal to the number of multichains 
$\pi_1\le \cdots\le\pi_{m-1}$ in $P$ (see Section~3.11 of 
\cite{Sta97}). From Proposition~\ref{modbij}, the number of such 
multichains in which at least one of the $\pi_{i}$ has positive connectivity 
is given by
\begin{equation}  \label{eqn:5.403}
\sum_{c\ge 1}2c\sum_{\substack{a_1,\ldots,a_{m-1},\\ b_1,\ldots ,b_{m-1}\ge 0}}
{p\choose a_1+\ldots +a_{m-1}+c}{q\choose b_1+\ldots +b_{m-1}-c}
\prod_{j=1}^{m-1}{p\choose a_j}{q\choose b_j}.
\end{equation}
The summation in~(\ref{eqn:5.403}) can be simplified 
if one invokes a well-known multinomial formula (which 
incidentally is a generalization of the binomial identity 
(\ref{vanderm}) used in the preceding sections), stating that
for any integers $A_1, \ldots , A_k, A_{k+1} \geq 0$ and 
$0 \leq b \leq A_{k+1}$ we have
\begin{equation}  \label{hypersum}
\sum_{a_1,\ldots ,a_k\ge 0}{A_1\choose a_1}\cdots {A_k\choose a_k}
{A_{k+1}\choose a_1+\cdots +a_k+b} = 
{A_1+\cdots +A_{k+1}\choose A_{k+1}-b}.
\end{equation}
By applying (\ref{hypersum}) to the inner summation, 
we find that~(\ref{eqn:5.403}) simplifies to
\begin{equation}  \label{eqn:5.405}
\sum_{c\ge 1}2c{mp\choose p-c}{mq\choose q+c}.
\end{equation}

On the other hand, we also have a simple formula for multichains
$\pi_1 \leq \cdots \leq \pi_{m-1}$ in which all of the $\pi_{i}$ have
connectivity equal to $0$. Indeed, these are simply multichains
in the direct product of $NC^{(B)}(p)$ with $NC^{(B)}(q)$, and 
thus the number of these is given by
\begin{equation}  \label{eqn:5.406}
Z_{\ncb (p)} (m) \cdot Z_{\ncb (q)} (m)
={mp\choose p}{mq\choose q},
\end{equation}
from Proposition~7 of~\cite{R97}. 

The expression for $Z_{\ncb (p,q)} (m)$ now follows
by adding together~(\ref{eqn:5.405}) and~(\ref{eqn:5.406}).

$2^o$ For any partially ordered set $P$, the number of
maximal chains is given by $d\, !$ times the coefficient of $m^d$ in 
$Z_P(m)$, where the zeta polynomial $Z_P(m)$ has degree $d$ (see
Proposition 3.11.1(a) of~\cite{Sta97}).
{}From part $1^o$ of this result, we have $d=p+q$ in this case, and the 
result follows from part $1^o$.
\end{proof}

\begin{cor} \label{cor:5.5}
For $p,q\ge 1$, $NC^{(B)}(p,q)$ has M\"obius function
$$\mu_{NC^{(B)}(p,q)}( \, \widehat{0},\widehat{1} \, )
=(-1)^{p+q} \left( {2p-1\choose p}{2q-1\choose q}+
 \sum_{c=1}^p 2c{2p-c-1\choose p-1}{2q+c-1\choose q-1}\right).$$
\end{cor}

\begin{proof}
This follows immediately from Theorem \ref{thm:5.4}, using the fact
that  for any partially ordered set $P$ one has 
$\mu_P( \, \widehat{0},\widehat{1} \, )=Z_P(-1)$ 
(see Proposition~3.11.1(c) of~\cite{Sta97}).
\end{proof}

\begin{rem} \label{rem:5.6}
It is straightforward to specialize Corollary~\ref{cor:5.5} to the 
case $p=n-1$, $q=1$, either by directly evaluating the summation 
or by setting $p=1$, $q=n-1$ and using the symmetry between $p$ 
and $q$. Using either of these means, we 
obtain the expression given in Theorem~\ref{thm:3.6}. Note further 
that we can specialize Theorem~\ref{thm:5.4} itself in the
same way, to obtain that for every $n \geq 2$, the zeta polynomial
of $\ncb (n-1,1)$ is given by the formula
\begin{equation}  \label{eqn:5.601}
Z_{ \ncb (n-1,1)} (m)
= \Bigl( 2 + \frac{mn}{(m-1)(n-1)} \Bigr) \cdot
{m(n-1)\choose n}.
\end{equation}
\end{rem}

\section{ The case  \boldmath{$\ncb (n_1,\ldots ,n_k)$}}\label{sec:6}
\setcounter{section}{6}

\setcounter{equation}{0}

In this section we consider the extension to multiannular 
non-crossing partitions of type $B$. The main point of the 
section is to establish that, due to a topological restriction 
called ``the genus inequality'',
the general multiannular case reduces in fact to the cases
of non-crossing partitions of type $B$ in a disc or an annulus.

We find it convenient to introduce the following notation.

\begin{notn} \label{notn:6.1}

$1^o$ For $\tau \in B_n$, let $\#(\tau)$
denote the number of orbits of $\tau$, and let $ii(\tau)$ denote 
the number of inversion-invariant orbits of $\tau$.

$2^o$ For $\tau,\sigma\in B_n$, let $\#(\tau, \sigma)$ denote
the number of orbits for the action of the subgroup of $B_n$ 
generated by $\{\tau, \sigma\}$.
\end{notn}

\begin{rem}\label{rem:6.2}
$1^o$ In terms of the notations introduced above, the formula
(\ref{eqn:2.2}) for the length
of an element $\tau \in B_n$ is now written in the form
\begin{equation}\label{ellii}
\ell_B (\tau)=n-\tfrac{1}{2} \bigl( \, \#(\tau)
- ii(\tau) \, \bigr) .
\end{equation}

$2^o$ The genus inequality mentioned at the beginning of the 
section is an inequality that arises in multiplying arbitrary 
permutations, and is stated as follows: for any permutations 
$\tau,\sigma$ of a finite set $X$, we have that 
\begin{equation}\label{genusiii}
\#(\sigma)+\#(\tau)+\#(\tau^{-1}\sigma)\leq 
\left| X\right| +2\cdot \#(\tau,\sigma).
\end{equation}
The name ``genus inequality'' comes from the fact that the 
difference of the right-hand side and left-hand side of 
(\ref{genusiii}) is a non-negative even integer $2g$, where 
$g$ is the genus for a certain orientable surface constructed 
from $\tau$ and $\sigma$ (see, \textit{e.g.}, Proposition 1.5.3 
of~\cite{LZ04}).
\end{rem}

\begin{defn}
Throughout the rest of the section we let $\gamma$ be a fixed 
element of $B_n$ with $\# (\gamma)=ii(\gamma )=k$, in which the 
$k$ orbits of $\gamma$ have sizes $2n_1,\ldots ,2n_k$, where 
$n_1,\ldots ,n_k\ge 1$ and $n_1+\cdots +n_k=n$.
Analogously to~(\ref{eqn:2.5}) and~(\ref{eqn:2.4}), we define
\begin{equation}  \label{eqn:5.2}
\sncb (n_1,\ldots ,n_k) :=  \{ \tau \in B_n \mid \tau \leq \gamma \} ,
\end{equation}
where the partial order $\leq$ is defined in~(\ref{eqn:2.3}).
We then define $\ncb (n_1,\ldots ,n_k)$ as in~(\ref{eqn:2.6}), by
putting
\begin{equation}
\ncb (n_1,\ldots ,n_k) := 
\{ \otilda ( \tau ) \mid \tau \in \sncb (n_1,\ldots ,n_k) \} ,
\end{equation}
where the adjusted orbit map $\otilda$ is as in Definition 
\ref{def:2.2}.
\end{defn}

\begin{prop} \label{multisize}
Let $\tau$ be a permutation in $\sncb (n_1,\ldots ,n_k)$. We denote
$\#(\tau,\gamma) =: m$ ($1 \leq m \le k$). Let $Y_1, \ldots , Y_m$ 
denote the orbits counted by $\# ( \tau , \gamma )$,
with $|Y_j|=2y_j$, $j=1,\ldots ,m$. Moreover, for every $1 \leq j \leq m$, let
us denote the restrictions of $\tau$ and $\gamma$ to $Y_j$ 
by $\tau_j$ and $\gamma_j$, respectively. We then have
\begin{equation}
\#(\gamma_j)=ii(\gamma_j)\le 2, \qquad j=1,\ldots ,m.
\end{equation}
\end{prop}

\begin{proof}
Note that $Y_j=X_j\cup-X_j$, for $j=1,\ldots ,m$, where
$\{ X_1,\ldots ,X_m\}$ is a partition of $\{ 1,\ldots ,n\}$ into 
nonempty subsets, with $|X_j| = y_j$ for $1 \leq j \leq m$.
Now, the triangle inequality~(\ref{triangineq}) gives
\begin{equation*}
\ell_B ( \gamma_j )  \leq\ell_B ( \tau_j ) + 
\ell_B ( \tau_j^{-1} \gamma_j ),\qquad j=1,\ldots ,m.
\end{equation*}
Then these inequalities, together with~(\ref{eqn:5.2}),~(\ref{eqn:2.3})
and the facts that 
$\ell_B(\gamma)$ $=\ell_B(\gamma_1)+\cdots +\ell_B(\gamma_m)$,
$\ell_B(\tau)$ $=\ell_B(\tau_1)+\cdots +\ell_B(\tau_m)$, give
\begin{equation}\label{pomparts}
\tau\in\sncb (n_1,\ldots ,n_k)\Rightarrow \ell_B ( \gamma_j ) 
= \ell_B ( \tau_j ) + \ell_B ( \tau_j^{-1} \gamma_j ),
\qquad j=1,\ldots ,m.
\end{equation}
But, from~(\ref{ellii}), rearranging the equation 
in~(\ref{pomparts}) and
using the fact that $\#(\gamma_j)=ii(\gamma_j)$, we obtain
\begin{equation}\label{poii}
\#( \gamma_j )+\#(\tau_j)+\#(\tau_j^{-1}\gamma_j) =
2y_j+ii(\gamma_j)+ii( \tau_j ) + ii( \tau_j^{-1} \gamma_j ),
\qquad j=1,\ldots ,m.
\end{equation}
On the other hand, the genus inequality~(\ref{genusiii}) implies that
\begin{equation}\label{genusii}
\#(\gamma_j)+\#(\tau_j)+\#(\tau_j^{-1}\gamma_j)\leq 2y_j+2,
\end{equation}
since $\#(\tau_j,\gamma_j)=1$. Thus~(\ref{poii}) and~(\ref{genusii}), 
together, imply that
\begin{equation*}
ii(\gamma_j)+ii( \tau_j ) + ii( \tau_j^{-1} \gamma_j )\le 2,
\end{equation*}
and, in particular, that $ii(\gamma_j)\le 2$, as required.
\end{proof}

\begin{rem}
When re-phrased in terms of partitions, Proposition~\ref{multisize} 
says that every partition $\pi \in \ncb (n_1,\ldots ,$ $n_k)$ splits
the $k$ orbits of $\gamma$ into groups of cardinality 1 or 2; thus
$\pi$ is obtained by putting together several separate ``pieces'',
where each piece is either a non-crossing partition of type $B$ in 
a disc, or a non-crossing partition of type $B$ in an annulus -- 
precisely the cases  that were considered earlier in the paper. 
Due to this phenomenon, the enumerative properties of 
$\ncb (n_1, \ldots , n_k)$ are quickly reduced to what we know 
from the disc and the annular cases, where one also has to do a 
suitable summation over partial matchings for the $k$ orbits 
of $\gamma$. For illustration, we finish with an example of how such a
calculation is carried out, in the particular case when $k=3$.
\end{rem}

\begin{exam} \label{exam:6.4}
Suppose that $k=3$, and we want to determine the total number of
partitions in $\ncb (n_1,n_2,n_3)$. By invoking 
Proposition~\ref{multisize}, taking into account that 
there are 4 possible partial matchings of the set of 3 orbits of 
$\gamma$, we find that 
\begin{eqnarray*}
\left|NC^{(B)}(n_1,n_2,n_3)\right|&=&
\left|\ncb (n_1)\right|\cdot\left|\ncb_{+}(n_2,n_3)\right|
+ \left|\ncb (n_2)\right|\cdot\left|\ncb_{+}(n_1,n_3)\right|\\
&{}&\!\!\!\!\!\!\!\!\!\!\!\!\!\!\!\!\!\!\!\!\!\!\!\!\!\!\!\!\!\!\!\!\!\!
+\left|\ncb (n_3)\right|\cdot\left|\ncb_{+}(n_1,n_2)\right|
+\left|\ncb (n_1)\right|\cdot\left|\ncb(n_2)\right|\cdot\left|\ncb(n_3)\right|,
\end{eqnarray*}
where $\ncb_{+}(p,q)$ denotes the subset of $\ncb(p,q)$ with positive 
connectivity. But Theorems~\ref{thm:4.4}.2 and~\ref{thm:4.4}.3 give
\begin{equation*}
\left|\ncb_{+}(p,q)\right|=\frac{pq}{p+q}{2p\choose p}{2q\choose q},
\end{equation*}
and from these results together with~(\ref{eqn:3.021}) we conclude that
\begin{equation*}
\left|NC^{(B)}(n_1,n_2,n_3)\right| =
\left(1+\frac{n_1n_2}{n_1+n_2}+\frac{n_1n_3}{n_1+n_3}+
\frac{n_2n_3}{n_3+n_3}  \right)
{2n_1\choose n_1}{2n_2\choose n_2}{2n_3\choose n_3}.
\end{equation*}
\end{exam}

$\ $

$\ $

\noindent
I.P. Goulden: University of Waterloo. 

\noindent
Address: Department of Combinatorics and Optimization, 
University of Waterloo,
Waterloo, Ontario N2L 3G1, Canada.

\noindent
Email: ipgoulden@math.uwaterloo.ca

$\ $

\noindent
Alexandru Nica, Ion Oancea: University of Waterloo.

\noindent
Address: Department of Pure Mathematics,
University of Waterloo,
Waterloo, Ontario N2L 3G1, Canada.

\noindent
Email: anica@math.uwaterloo.ca, ioancea@alumni.uwaterloo.ca

$\ $

\noindent
Current address for Ion Oancea: OANDA Canada,
370 King St. W., 2nd Floor, Box 60,
Toronto, Ontario M5V 1J9, Canada.

\noindent
Email: ioancea@oanda.com

\end{document}